\documentclass[UTF-8,reqno]{amsart}
\usepackage{enumerate}
\usepackage{amssymb,url,color, booktabs}

\usepackage{mathrsfs}
\usepackage{amsmath}

\usepackage{fancyhdr}
\usepackage[nobysame]{amsrefs}
\BibSpec{article}{%
+{}{\PrintAuthors} {author}
+{,}{ \textrm} {title}
+{.}{ \textit} {journal}
+{,}{ \textbf} {volume}
+{}{ \parenthesize} {date}
+{,}{ } {pages}
+{.}{ arXiv:} {eprint}
+{.}{} {transition}
}
\BibSpec{book}{%
+{}{\PrintAuthors} {author}
+{,}{ \textit} {title}
+{.}{ \textrm} {series} 
+{,}{ Vol.} {volume}
+{.}{ } {publisher}
+{,}{ } {date}
+{.}{} {transition}
}
\usepackage{color}
\usepackage[colorlinks=true]{hyperref}
\hypersetup{
    linkcolor=blue,          
    citecolor=red,        
    filecolor=blue,      
    urlcolor=cyan
}

\numberwithin{equation}{section}

\newcommand{\be}{\begin{eqnarray}}
\newcommand{\ee}{\end{eqnarray}}
\newcommand{\ce}{\begin{eqnarray*}}
\newcommand{\de}{\end{eqnarray*}}
\newtheorem{theorem}{Theorem}[section]
\newtheorem{lemma}[theorem]{Lemma}
\newtheorem{remark}[theorem]{Remark}
\newtheorem{definition}[theorem]{Definition}
\newtheorem{proposition}[theorem]{Proposition}
\newtheorem{Examples}[theorem]{Example}
\newtheorem{corollary}[theorem]{Corollary}
\newtheorem{assumption}[theorem]{Assumption}

\newenvironment{proof of theorem 1.2 and 1.3}{{\it Proof of Theorem 1.2 and 1.3}.}{{\hfill 	
$\square$\hskip - \parfillskip}}
\newenvironment{proof of theorem 1.4}{{\it Proof of Theorem 1.4}.}{{\hfill 	
		$\square$\hskip - \parfillskip}}
\newenvironment{proof of corollaries 6.2 and 6.3}{{\it Proof of Corollaries \ref{c6.2} and \ref{c6.3}}.}{{\hfill 	
		$\square$\hskip - \parfillskip}}
\newenvironment{proof of theorem 3.1 and 3.2}{{\it Proof of Theorems \ref{t3.1} and \ref{t3.2}}.}{{\hfill 	
		$\square$\hskip - \parfillskip}}

\makeatletter
\newcommand{\rmnum}[1]{\romannumeral #1}
\newcommand{\Rmnum}[1]{\expandafter\@slowromancap\romannumeral #1@}
\makeatother

\def\bg{\bar G}
\def\wg{\widetilde G}

\def\eps{\varepsilon}

\def\e{\mathrm{e}}

\def\a{\alpha}
\def\om{\omega}

\def\p{\partial}

\def\g{\gamma}
\def\l{\lambda}
\def\la{\langle}
\def\ra{\rangle}
\def\[{{\Big[}}
\def\]{{\Big]}}
\def\<{{\langle}}
\def\>{{\rangle}}
\def\({{\Big(}}
\def\){{\Big)}}

\def\bx{{\mathbf{x}}}

\def\min{{\mathord{{\rm min}}}}

\def\={&\!\!=\!\!&}

\def\cL{{\mathcal L}}

\def\mR{{\mathbb R}}
\def\mS{{\mathbb S}}

\def\1{{\mathbf{1}}}

\def\sF{{\mathscr F}}

\def\geq{\geqslant}
\def\leq{\leqslant}
\def\ge{\geqslant}
\def\le{\leqslant}

\def\k{\kappa}

\def\eps{\varepsilon}

\def\e{\mathrm{e}}

\def\a{\alpha}
\def\om{\omega}

\def\p{\partial}

\def\g{\gamma}
\def\l{\lambda}
\def\la{\langle}
\def\ra{\rangle}
\def\[{{\Big[}}
\def\]{{\Big]}}
\def\<{{\langle}}
\def\>{{\rangle}}
\def\({{\Big(}}
\def\){{\Big)}}

\def\bx{{\mathbf{x}}}

\def\min{{\mathord{{\rm min}}}}

\def\={&\!\!=\!\!&}
\def\bt{\begin{theorem}}
\def\et{\end{theorem}}
\def\bl{\begin{lemma}}
\def\el{\end{lemma}}
\def\br{\begin{remark}}
\def\er{\end{remark}}
\def\bx{\begin{Examples}}
\def\ex{\end{Examples}}
\def\bd{\begin{definition}}
\def\ed{\end{definition}}
\def\bp{\begin{proposition}}
\def\ep{\end{proposition}}
\def\bc{\begin{corollary}}
\def\ec{\end{corollary}}

\def\geq{\geqslant}
\def\leq{\leqslant}
\def\ge{\geqslant}
\def\le{\leqslant}

 \def\nn{\nabla}

\def\<{\langle} \def\>{\rangle}

\def\bpf{\begin{proof}}
\def\epf{\end{proof}}

\allowdisplaybreaks

\begin{document}
	
\title{A flow method for curvature equations}
\thanks{\it {The research is partially supported by NSFC (Nos. 11871053 and 12261105).}}
\author{Shanwei Ding and Guanghan Li}


\address{School of Mathematics and Statistics, Wuhan University, Wuhan 430072, China.
}

\begin{abstract}
We consider a general curvature equation $F(\k)=G(X,\nu(X))$, where $\k$ is the principal curvature of the hypersurface $M$ with position vector $X$. It includes the classical prescribed curvature measures problem and area measures problem. However, Guan-Ren-Wang \cite{GRW} proved that the $C^2$ estimate fails usually for general function $F$. Thus, in this paper, we pose some additional conditions of $G$ to get existence results by a suitably designed parabolic flow. In particular, if  $F=\sigma_{k}^\frac{1}{k}$ for $\forall 1\le k\le n-1$, the existence result has been derived in the famous work \cite{GLL} with $G=\psi(\frac{X}{|X|})\la X,\nu\ra^{\frac1k}{|X|^{-\frac{n+1}{k}}}$. This result will be generalized to $G=\psi(\frac{X}{|X|})\la X,\nu\ra^\frac{{1-p}}{k}|X|^\frac{{q-k-1}}{k}$ with $p>q$ for arbitrary $k$ by a suitable auxiliary function. 
The uniqueness of the solutions in some cases is also studied.


\end{abstract}

\maketitle
\setcounter{tocdepth}{2}
\tableofcontents

\section{Introduction}
This paper concerns a longstanding problem of the curvature equation in general form
\begin{equation}\label{1.1}
F(\k(X))=G(X,\nu(X)) \quad \forall X\in M,
\end{equation}
where $F$ satisfies Assumption \ref{a1.3}, $\nu(X)$ and $\k(X)$ are the outer normal and principal curvatures of hypersurface $M\subset \mR^{n+1}$ at $X$, respectively.

The $k$th elementary symmetric function $\sigma_k$ is defined by
$$\sigma_k(\k_1,...,\k_n)=\sum_{1\leq i_1<\cdot\cdot\cdot<i_k\leq n}\k_{i_1}\cdot\cdot\cdot\k_{i_k},$$
and let $\sigma_0=1$.

Equation (\ref{1.1}) is associated with many important geometric problems. In particular, if $F(\k)=\sigma_{k}^\frac{1}{k}$ and $F(\k)=(\frac{\sigma_{n}}{\sigma_{n-k}})^\frac{1}{k}$, this equation will convert into the corresponding Christoffel-Minkowski problem for the general curvature measures and general area measures respectively. More detailed studies can be found in \cite{GLL,SR}.

If $F(\k)=\sigma_{k}^\frac{1}{k}$, equation (\ref{1.1}) has been studied extensively; it is a special type of general equation systematically studied by Alexandrov in \cite{AA}. With the following appropriate Assumption \ref{a1.1} and  Assumption \ref{a1.2} on the function $G$ as considered in \cite{CNS2,TW}, one may establish the $C^0$ and $C^1$ estimates of (\ref{1.1}) with suitable $G$.
\begin{assumption}\label{a1.1}
Suppose $G\in C^2(\mR^{n+1}\times\mS^n)$ is a positive function if $\la X,\nu\ra>0$, and
there are two positive constants $r_1<r_2$ such that
\begin{equation}\label{1.2}
\begin{cases}
G(X,\frac{X}{|X|})\ge\frac{F(1,\cdots,1)}{r_1} \quad \text{ for }|X|=r_1,\\[3mm]
G(X,\frac{X}{|X|})\le\frac{F(1,\cdots,1)}{r_2} \quad \text{ for }|X|=r_2.
\end{cases}
\end{equation}
\end{assumption}
\begin{assumption}\label{a1.2}
	Suppose $G\in C^2(\mR^{n+1}\times\mS^n)$ is a positive function and
for any fixed unit vector $\nu$,
\begin{equation}\label{1.3}
\frac{\p}{\p\rho}(\rho G(X,\nu))\le0\quad\text{  for }r_1<\rho<r_2, \text{ where }|X|=\rho.
\end{equation}
\end{assumption}
It is of great interest to establish the $C^2$-estimate for equation (\ref{1.1}) both in geometry and in PDEs. $C^2$-estimates are known in many special cases. For example, when $F(\k)=\sigma_{k}^\frac{1}{k}$, equation (\ref{1.1}) is quasi-linear for $k=1$, and the $C^2$-estimate follows from the classical theory of quasi-linear PDEs. The equation is of Monge-Amp$\grave{e}$re type if $k=n$, and the $C^2$-estimate in this case for general $G(X,\nu)$ is due to Caffarelli-Nirenberg-Spruck \cite{CNS}. For $G$ independent of the normal vector $\nu,$ the $C^2$-estimate was proved by Caffarelli-Nirenberg-Spruck \cite{CNS2} for a general class of fully nonlinear operators $F$, including  $F(\k)=\sigma_{k}^\frac{1}{k}$ and $F(\k)=(\frac{\sigma_{k}}{\sigma_{l}})^\frac{1}{k-l}$. For $G$ in (\ref{1.1}) dependent only on $\nu$, the $C^2$-estimate was proved in \cite{GG}. The $C^2$-estimate was also proved for equations of the prescribed curvature measures problem in \cite{GLL,GLM}, where $G(X,\nu)=\la X,\nu\ra \tilde{G}(X)$. Here $\la\cdot,\cdot\ra$ is the standard inner product in $\mR^{n+1}$.

Recently Guan-Ren-Wang \cite{GRW} obtained the global curvature estimate of the closed convex hypersurface and the star-shaped $2$-convex hypersurface. Li-Ren-Wang \cite{LRW} substituted the convex condition by $(k+1)$-convex condition for any Hessian equations and derived the Pogorelov type interior $C^2$-estimate. For the case $k=n-1$ and $k=n-2$, Ren-Wang \cite{RW1,RW2} completely solved the longstanding problem, that they obtained the global curvature estimates of $n-1$ and $n-2$ convex solutions for $n-1$ and $n-2$ Hessian equations. Chen-Li-Wang \cite{CLW} established the global curvature estimate for the prescribed curvature problem in warped product spaces. Li-Ren-Wang \cite{LRW} considered the global curvature estimate of convex solutions for a class of general Hessian equations. Spruck-Xiao \cite{SX} established the curvature estimate for the prescribed scalar curvature problem in space forms and gave a simple proof of $C^2$-estimate for the same equation in the Euclidean space.

If $F(\k)=(\frac{\sigma_{n}}{\sigma_{n-k}})^\frac{1}{k}$, the hypersurface $M$ is naturally strictly convex. At this time, (\ref{1.1}) can be converted into
\begin{equation}\label{1.4}
\sigma_{k}(\l)=G(X,\nu),
\end{equation}
where $\l_i=\frac{1}{\k_i}$ are the principal curvature radii of hypersurface. In particular, (\ref{1.4}) is $L^p$ dual Christoffel-Minkowski problem posed in \cite{DL3} if $G=\psi(\nu)\la X,\nu\ra^{p-1}|X|^{k+1-q}$
. If $k=n$ or $q=k+1$, it will be converted into $L^p$ dual Minkowski problem or $L^p$ Christoffel-Minkowski problem respectively. If $k=n$ and $G=\psi(\nu)\varphi(\la X,\nu\ra)\phi(X)$, equation (\ref{1.4}) is dual Orlicz Minkowski problem posed in \cite{GHW1,GHW2}. More beautiful examples and previous known results of these
problems can be seen in \cite{LE,LO,LYZ2,HLY2,LL,GM,HMS,GX,CL,HZ,GHW1,GHW2,CCL,CHZ,DL3,DL4} etc..

Before stating the main theorems, we would like to recall the existing counterexamples. If $F=(\frac{\sigma_{k}}{\sigma_{l}})^\frac{1}{k-l}$ in (\ref{1.1}), where $1\le l<k\le n$, Guan-Ren-Wang \cite{GRW} proved that the $C^2$-estimate fails for strictly convex hypersurface. Furthermore, if $F(\k)=(\frac{\sigma_{n}}{\sigma_{n-k}})^\frac{1}{k}(\k)$ and $G(X,\nu)=\psi(\nu)\la X,\nu\ra^\frac{1-p}{k}$, Ivaki \cite{IM} proved that a convexity condition of $G$ is essential to ensure that a corresponding flow smoothly converges to a solution of $L^p$ Christoffel-Minkowski problem. Thus, in order to obtain the $C^2$-estimate of (\ref{1.1}) we believe that it is necessary to add some conditions to $G$.

In order to obtain existence results we make the following assumption. In this paper, if there is no special statement, we say general $F$ to mean a function $F$ satisfies Assumption \ref{a1.3}. We call hypersurface $M$ is admissible hypersurface if its principal curvatures $\k(M)\in\Gamma$, where $\Gamma$ is defined in the following assumption.
\begin{assumption}\label{a1.3}
	Let $\Gamma\subseteq\mathbb{R}^n$ be a symmetric, convex, open cone containing
	\begin{equation}\label{1.5}
		\Gamma_+=\{(\k_i)\in\mathbb{R}^n:\k_i>0\},
	\end{equation}
	and suppose that $F$ is positive in $\Gamma$, homogeneous of degree $1$, and concave with
	\begin{equation}\label{1.6}
		\frac{\p F}{\p\k_i}>0,\quad F\vert_{\p\Gamma}=0,\quad F^{-\beta}(1,\cdots,1)=\eta.
	\end{equation}
\end{assumption}
Recall Assumption \ref{a1.2}. There are many examples whose are positive functions in $\mR^{n+1}\times\mS^n$, such as $\psi(\nu)\phi(X)$, $\phi(X^T)$, $\phi(\nu,\rho)$ etc., where $X^T=X-\la X,\nu\ra \nu$. However, if $G(x,y)=\la x,y\ra^\a\phi(x)$, $G$ may not preserve positive in $\mR^{n+1}\times\mS^n$, where $\phi(x)$ is a positive function. In fact, if $\la x_0,y_0\ra>0$ and $\a$ is odd, $G(-x_0,y_0)<0$. Hence, we hope to consider the following functions, which is only required to some weaker conditions. We let 
$G(X,\nu)$ be a positive function in $\mR^{n+1}\times\mS^n$ if $\la X,\nu\ra>0$, which has been considered in \cite{DL4}. However, if the hypersurface $M$ is not convex, the bounds of this function and its derivatives are hard to be derived even if we have $C^0$-estimate. As an alternative, we consider $G=\phi(X,\nu)\varphi(\la X,\nu\ra)$ to derive the $C^1$-estimate, where $\phi$ is a positive functions in $\mR^{n+1}\times\mS^n$, $\varphi:\mR_+\rightarrow\mR_+$. Obviously $G=\phi(X,\nu)\varphi(\la X,\nu\ra)$ is a positive function if $\la X,\nu\ra>0$.

Compared with Assumption \ref{a1.2}, we hope to derive the $C^1$-estimate in a wider case and therefore make the following assumptions. 
In this paper, if there is no special statement, $G$ refers to a positive function in $\mR^{n+1}\times\mS^n$ if $\la X,\nu\ra>0$.
\begin{assumption}\label{a1.4}
Let $G=\phi(X,\nu)\varphi(\la X,\nu\ra)$ be a smooth function, where $\phi$ is a positive functions in $\mR^{n+1}\times\mS^n$, $\varphi:\mR_+\rightarrow\mR_+$. We pose the following conditions.

(\rmnum{1}) $\varphi=1$, and
for any fixed unit vector $\nu$, if $\la X,\nu\ra>0$,
\begin{equation*}
	\frac{\p}{\p\rho}\(\rho \phi(X,\nu)\)\le0\quad\text{  for }r_1<\rho<r_2, \text{ where }|X|=\rho;
\end{equation*}

(\rmnum{2}) for any $u=\la X,\nu\ra>0$, there exists a positive constant $\eps$, s.t.

$$u(\log\varphi)'=\frac{u\varphi_u}{\varphi}\le-\eps,\;\;
{\mbox{where}} \;\;\varphi_u=\frac{\partial \varphi}{\p u};$$

(\rmnum{3}) for any fixed unit vector $\nu$, if $\la X,\nu\ra>0$, there exists a positive constant $\eps$, s.t.
\begin{equation*}
\frac{\p}{\p\rho}\(\rho \phi(X,\nu)\varphi(\la X,\nu\ra)\)\le-\eps G\quad\text{  for }r_1<\rho<r_2, \text{ where }|X|=\rho;
\end{equation*}

(\rmnum{4}) $\varphi=1$, and
for any fixed unit vector $\nu$, if $\la X,\nu\ra>0$,
\begin{equation*}
	\frac{\p}{\p\rho}\(\rho \phi(X,\nu)\)<0\quad\text{  for }r_1<\rho<r_2, \text{ where }|X|=\rho;
\end{equation*}

(\rmnum{5}) $\varphi=1$, $\phi(X,\nu)=\phi(X)$ and
for any fixed unit vector $\nu$, if $\la X,\nu\ra>0$,
\begin{equation*}
	\frac{\p}{\p\rho}\(\rho \phi(X)\)\le0\quad\text{  for }r_1<\rho<r_2, \text{ where }|X|=\rho;
\end{equation*}

(\rmnum{6}) for any $u=\la X,\nu\ra>0$, there exists a positive constant $\eps$, s.t.

$$|u(\log\varphi)'|=|\frac{u\varphi_u}{\varphi}|\ge\eps.$$
\end{assumption}



We state the first result of this paper, which is proved by flow method. We mention that the existence to solutions of this kind problems is usually proved by elliptic method in previous papers.
\begin{theorem}\label{t1.5}
 Let $F$ satisfy Assumption \ref{a1.3} and $G(X,\nu)$ satisfy Assumption \ref{a1.1}. Suppose $G$ satisfies one of (\rmnum{2})-(\rmnum{5}) of Assumption \ref{a1.4}, and if

(1) $d_\nu G(\nu)\le G$;

(2) for any fixed vector $X$, $G$ is strictly log-convex with respect to $\nu$,

i.e. $(\frac{\p^2}{\p\nu^\a\p\nu^\beta}(\log G(X,\nu)))$ is positive definite.

\noindent Then equation (\ref{1.1}) has a star-shaped, smooth admissible  solution $M$ in $\{r_1\le|X|\le r_2\}$.
\end{theorem}
\textbf{Remark:} (1) Given the $C^0$ and $C^1$ estimates, we can derive the $C^2$-estimate by conditions (1) and (2) in Theorem \ref{t1.5} if $G(X,\nu)\in C^\infty(B_{r_2}\backslash B_{r_1}\times\mS^n)$ is a positive function with $\la X,\nu\ra>0$. 

(2) We introduce the concept of $\a$-convex (or $\a$-concave) (see \cite{KA}). $f$ is said to be $\a$-convex (or $\a$-concave) for $0<\a<+\infty$ when $f^\a$ is convex (or concave), for $\a=0$ when $\ln f$ is  convex (or concave), for $-\infty<\a<0$ when $f^\a$ is concave (or convex). In other words, $f(x_1,\cdots,x_n)$ is said to be $\a$-convex if $\frac{\p^2 f}{\p x_i\p x_j}+(\a-1)\frac{\p_{x_i}f\p_{x_j}f}{f}$ is semi-positive definite by a simple calculation. We only posed a typical condition in this theorem.
In fact, the condition $G$ is strictly log-convex can be replaced by one of  the following conditions: (a) there exists $\a<\frac{1}{\beta+1}$ and $\beta>0$ such that $uG^\beta$ is $\a$-convex; (b) there exists $\beta>0$ such that  $uG^\beta$ is strictly $\frac{1}{\beta+1}$-convex.

In particular, if $F=\sigma_{k}^\frac{1}{k}$, the existence of solution to (\ref{1.1}) is still open for general $G$ except $k=1,2,n-2,n-1,n$. We give an existence result with some additional conditions of $G$ for arbitrary $2\le k\le n-1$ since the cases $k=1,n$ are easy to be dealt with.
\begin{theorem}\label{t1.6}
Let $F=\sigma_{k}^\frac{1}{k}$ where $2\le k\le n-1$, and $G(X,\nu)$ satisfy Assumption \ref{a1.1}. Suppose $G$ satisfies one of Assumption \ref{a1.4}, $d_\nu G(\nu)\le G$ and if either

(1) for any fixed vector $X$, $\frac{G^{k}}{\la X,\nu\ra}$ is log-convex with respect to $\nu$,

i.e. $(\frac{\p^2}{\p\nu^\a\nu^\beta}(\log\frac{G^{k}}{\la X,\nu\ra}))$ is semi-positive definite; 

or (2) there exists $0<\beta<1$ such that $uG^{-\beta}$ is $\frac{1}{1-\beta}$-concave.

\noindent Then equation (\ref{1.1}) has a star-shaped, smooth, strictly $k$-convex solution $M$ in $\{r_1\le|X|\le r_2\}$.
\end{theorem}
\textbf{Remark:} (1) If $k=n$, an existence result can be seen in \cite{DL4}, where we only need a $C^0$ condition to derive the existence result.

(2) As in Theorem \ref{1.5}, the $\frac{1}{1-\beta}$-concavity of $uG^{-\beta}$ is just a typical condition. In fact, for $0<\beta<1$ or $\beta\ge k$, if $\exists\a>\frac{1}{1-\beta}$ such that $uG^{-\beta}$ is $\a$-concave, we can also derive this result.

In order to better understand the above conditions, we consider the homogeneous case which is first considered in \cite{GLL,GLM} if $F=\sigma_{k}^\frac{1}{k}$, i.e.
\begin{equation}\label{1.7}
\sigma_{k}=\psi\(\dfrac{X}{\rho}\)u^{1-p}\rho^{q-k-1},
\end{equation}
where $u=\la X,\nu\ra$, $\rho=|X|$ are the support and radial functions of hypersurface $M$ at $X$ respectively.
At this time $G=\psi(\dfrac{X}{\rho})^{\frac1k}u^\frac{{1-p}}{k}\rho^\frac{{q-k-1}}{k}$.
Then (\ref{1.2}) and (\ref{1.3}) imply $p>q$ and $1=p\ge q$ respectively at this case. Conditions (\rmnum{1}), (\rmnum{2}), (\rmnum{3}) and (\rmnum{4}) of Assumption \ref{a1.4} imply $1=p\ge q$, $p<1$, $p>q$ and $p>q$, respectively.
(1) of Theorem \ref{t1.5} is equivalent to $p\ge1-k$. (2) of Theorem \ref{t1.5} and (1) of Theorem \ref{t1.6} mean $p>1$ and $p\ge0$ respectively. In summary, we derive the existence solutions of (\ref{1.7}) for (1) $p>q,p\ge0$, (2) $p=1-k>q$ with $F=\sigma_{k}^\frac{1}{k}$; and for $p>q,p>1$ with general $F$. Thus Theorem \ref{t1.6} extends the famous result \cite{GLL} which is the case $p=0,q=k-n$, $F=\sigma_{k}^\frac{1}{k}$. In details, we have the following corollary by Theorems \ref{t1.6} and \ref{t6.1}.
\begin{corollary}
Let $n\ge2$, $2\le k\le n-1$, and $\psi(\xi)\in C^\infty(\mS^n)$ be a positive function. If $p>q$ and either $p\ge0$,  or $p=1-k$,
 equation (\ref{1.7}) has a unique, star-shaped, smooth, strictly $k$-convex solution $M$ in $\{r_1\le|X|\le r_2\}$.
\end{corollary}


We call a hypersurface strictly $k$-convex if $\k\in\Gamma_k^+$ everywhere on $M$, where
$$\Gamma_k^+ =\{(\k_i)\in\mathbb{R}^n:\sigma_l>0 \mbox {, } \forall1\le l\le k\le n \}.$$

Let us make some remarks about our conditions. The convex cone $\Gamma$ that contains the positive cone in (\ref{1.5}) is decided by $F$, e.g., $\Gamma =\Gamma_1^+$ if $F=\sigma_1$; $\Gamma =\Gamma_k^+$ if $F=\sigma_k^{\frac 1k}$; $\Gamma =\Gamma_+$ the positive cone if $F=\sigma_n^{\frac 1n}$. (\ref{1.6}) ensures that this equation is elliptic.  

We give some examples of functions $F$ satisfying the required hypotheses. For any integer $k,l$ such that $0\leq k< l\leq n$, $(\frac{\sigma_l}{\sigma_k})^{\frac{1}{l-k}}$ is smooth, positive, symmetric function  and homogenous of degree one on the $l$-convex cone. It is easy to check that (\ref{1.2}) and (\ref{1.3}) hold for $(\frac{\sigma_l}{\sigma_k})^{\frac{1}{l-k}}$. $(\frac{\sigma_l}{\sigma_k})^{\frac{1}{l-k}}$ satisfies the concavity by \cite{HGC}.

The second example is $F=(\sum_{i=1}^n\k_i^{k})^{\frac{1}{k}}$ for $k<0$.
What's more,
if $F_1,\cdots,F_k$ satisfy our conditions, then $F=\prod_{i=1}^kF_i^{\alpha_i}$ also satisfies our conditions, where $\alpha_i\ge 0$ and $\sum_{i=1}^k\alpha_i=1$.
More examples can be seen in \cite{B3,B4}.

We shall prove the above theorems by a flow method. In Section 3, we will explain why the asymptotic behavior of these flows mean existence of solutions of the curvature equation (\ref{1.1}), and therefore complete the proof of main results. Thus in order to prove the above theorems, we shall establish the a priori estimates for some flows.

The rest of the paper is organized as follows. We first recall some notations and known results in Section 2 for later use. In Section 3, we  briefly describe the background of anisotropic flows firstly. Afterwards,
we state the existence and convergence result about some flows and derive the evolution equations. The $C^0$ and $C^1$ estimates of general flow (\ref{1.8}) are obtained in Section 4.
In Section 5, we derive the $C^2$-estimate and show
the convergence of these flows. 
At last, we shall prove the uniqueness results for some cases in Section 6.

\section{Preliminary}
We now state some general facts about hypersurfaces, especially those that can be written as graphs. The geometric quantities of ambient spaces will be denoted by $(\bar{g}_{\alpha\beta})$, $(\bar{R}_{\alpha\beta\gamma\delta})$ etc., where Greek indices range from $0$ to $n$. Quantities for $M$ will be denoted by $(g_{ij})$, $(R_{ijkl})$ etc., where Latin indices range from $1$ to $n$. We use the similar notations as those in \cite{DL2}.

Let $\nabla$, $\bar\nabla$ and $D$ be the Levi-Civita connection of $g$, $\bar g$ and the Riemannian metric $e$ of $\mathbb S^n$  respectively. All indices appearing after the semicolon indicate covariant derivatives. The $(1,3)$-type Riemannian curvature tensor is defined by
\begin{equation}\label{2.1}
	R(U,Y)Z=\nabla_U\nabla_YZ-\nabla_Y\nabla_UZ-\nabla_{[U,Y]}Z.
\end{equation}
We also denote the $(0,4)$ version of the curvature tensor by $R$,
\begin{equation}\label{2.4}
	R(W,U,Y,Z)=g(R(W,U)Y,Z).
\end{equation}
The induced geometry of $M$ is governed by the following relations. The second fundamental form $h=(h_{ij})$ is given by the Gaussian formula
\begin{equation}\label{2.5}
	\bar\nabla_ZY=\nabla_ZY-h(Z,Y)\nu,
\end{equation}
where $\nu$ is a local outer unit normal field. Note that here (and in the rest of the paper) we will abuse notation by disregarding the necessity to distinguish between a vector $Y\in T_pM$ and its push-forward $X_*Y\in T_p\mathbb{R}^{n+1}$. The Weingarten endomorphism $A=(h_j^i)$ is given by $h_j^i=g^{ki}h_{kj}$, and the Weingarten equation
\begin{equation}\label{2.7}
	\nu_{;i}^\alpha=h_i^kX_{;k}^\alpha,
\end{equation}
holds there, where we use the summation convention (and will henceforth do so).
We also have the Codazzi equation in $\mathbb{R}^{n+1}$
\begin{equation}\label{2.9}
	h_{ij;k}=h_{ik;j},
\end{equation}
and the Gauss equation
\begin{equation}\label{2.11}
	R_{ijkl}=h_{il}h_{jk}-h_{ik}h_{jl}.
\end{equation}

Let $dz^2$ be the standard metric on $\mathbb S^n$, then the metric of Euclidean space $\mathbb{R}^{n+1}$ in polar coordinate system is given by
\begin{equation*}
	\bar{g}:=ds^2=d\rho^2+\rho^2dz^2.
\end{equation*}
It is known that $u:=\la X,\nu\ra$ is the support function of a hypersurface in $\mathbb{R}^{n+1}$, where $\la\cdot,\cdot\ra=\bar{g}(\cdot,\cdot)$. Then we have the gradient and hessian of the support function $u$ under the induced metric $g$ on $M$.
\begin{lemma}\cite{GL,JL}\label{l2.2}
	The support function $u$ satisfies
	\begin{equation}\label{2.13}
		\begin{split}
			\nabla_iu=&g^{kl}h_{ik}\nabla_l\Phi,  \\
			\nabla_i\nabla_ju=&g^{kl}\nabla_kh_{ij}\nabla_l\Phi+\phi'h_{ij}-(h^2)_{ij}u,
		\end{split}
	\end{equation}
	where $(h^2)_{ij}=g^{kl}h_{ik}h_{jl}$ and
$\Phi(\rho)=\frac{1}{2}\rho^2.$
\end{lemma}

 For a hypersurface $(M,g)$ in $\mathbb{R}^{n+1}$, which is a graph of a smooth and positive function $\rho(z)$ on $\mathbb{S}^n$, let $\p_1,\cdots,\p_n$ be a local frame along $M$ and $\p_\rho$ be the vector field along radial direction. Then the support function, induced metric, inverse metric matrix, second fundamental form can be expressed as follows (\cite{GL}).
\begin{align*}
	u &= \frac{\rho^2}{\sqrt{\rho^2+|D\rho|^2}},\;\; \nu=\frac{1}{\sqrt{1+\rho^{-2}|D\rho|^2}}(\frac{\p}{\p\rho}-\rho^{-2}\rho_i\frac{\p}{\p x_i}),   \\
	g_{ij} &= \rho^2e_{ij}+\rho_i\rho_j,  \;\;   g^{ij}=\frac{1}{\rho^2}(e^{ij}-\frac{\rho^i\rho^j}{\rho^2+|D\rho|^2}),\\
	h_{ij} &=\(\sqrt{\rho^2+|D\rho|^2}\)^{-1}(-\rho D_iD_j\rho+2\rho_i\rho_j+\rho^2e_{ij}),\\
	h^i_j &=\frac{1}{\rho^2\sqrt{\rho^2+|D\rho|^2}}(e^{ik}-\frac{\rho^i\rho^k}{\rho^2+|D\rho|^2})(-\rho D_kD_j\rho+2\rho_k\rho_j+\rho^2e_{kj}),
\end{align*}
where $e_{ij}$ is the standard spherical metric. It will be convenient if we introduce a new variable $\gamma$ satisfying $$\frac{d\gamma}{d\rho}=\frac{1}{\rho}.$$
Let $\omega:=\sqrt{1+|D\gamma|^2}$, one can compute the unit outward normal $$\nu=\frac{1}{\omega}(1,-\frac{\gamma_1}{\rho},\cdots,-\frac{\gamma_n}{\rho})$$ and the support function $u=\frac{\rho}{\omega}$. Moreover,
\begin{align}\label{2.14}
	g_{ij} &=\rho^2(e_{ij}+\gamma_i\gamma_j), \;\; g^{ij}=\frac{1}{\rho^2}(e^{ij}-\frac{\gamma^i\gamma^j}{\omega^2}),\notag\\
	h_{ij} &=\frac{\rho}{\omega}(-\gamma_{ij}+\gamma_i\gamma_j+e_{ij}),\notag\\
	h^i_j &=\frac{1}{\rho\omega}(e^{ik}-\frac{\gamma^i\gamma^k}{\omega^2})(-\gamma_{kj}+\gamma_k\gamma_j+e_{kj})\notag\\
	&=\frac{1}{\rho\omega}(\delta^i_j-(e^{ik}-\frac{\gamma^i\gamma^k}{\omega^2})\gamma_{kj}).
\end{align}
Covariant differentiation with respect to the spherical metric is denoted by indices.

There is also a relation between the second fundamental form and the radial function on the hypersurface. Let $\widetilde{h}=\rho e$. Then
\begin{equation}\label{2.15}
	\omega^{-1}h=-\nn^2\rho+\widetilde{h}
\end{equation}
holds; cf. \cite{GC2}. Since the induced metric is given by
$$g_{ij}=\rho^2 e_{ij}+\rho_i\rho_j, $$
we obtain
\begin{equation}\label{2.16}
	\omega^{-1}h_{ij}=-\rho_{;ij}+\frac{1}{\rho}g_{ij}-\frac{1}{\rho}\rho_{i}\rho_{j}.
\end{equation}

A connection between $\vert\nn\rho\vert$ and $\vert D\gamma\vert$ can be seen in \cite{DL2}, i.e.
if $M$ is a star-shaped hypersurface, there holds
\begin{equation}\label{l2.3}
\vert\nn\rho\vert^2=1-\frac{1}{\omega^2}.
\end{equation}

Lastly,
we review some properties of elementary symmetric functions. See \cite{HGC} for more details.
In Section 1 we give the definition of elementary symmetric functions. The definition can be extended to symmetric matrices. Let $A\in Sym(n)$ be an $n\times n$ symmetric matrix. Denote by $\k=\k(A)$ the eigenvalues of $A$. Set $p_m(A)=p_m(\k(A))$. We have
\begin{equation*}
p_m(A)=\frac{(n-m)!}{n!}\delta^{j_1\cdots j_m}_{i_1\cdots i_m}A_{i_1j_1}\cdots A_{i_mj_m}, \qquad m=1,\cdots,n.
\end{equation*}
\begin{lemma}\label{l2.5}
If $\k\in\Gamma_m^+=\{x\in\mR^n: p_i(x)>0, i=1,\cdots,m\}$, we have the following Newton-MacLaurin inequalities for $1\le k\le m$.
\begin{eqnarray}
p_{m+1}(\k)p_{k-1}(\k)&\le& p_k(\k)p_m(\k),\\
p_1\ge p_2^{\frac{1}{2}}&\ge&\cdots\ge p_m^{\frac{1}{m}}.
\end{eqnarray}
Equality holds if and only if $\k_1=\cdots=\k_n$.
\end{lemma}
We assume that $\k_1\ge\cdots\ge\k_n$. Let us denote by $\sigma_{k,i}(\k)$ the sum of the terms of $\sigma_k(\k)$ not containing the factor $\k_i$. Then the following identities hold.
\begin{proposition}\label{p2.6}
\cite{HGC} We have, for any $k=0,\cdots,n$, $i=1,\cdots,n$ and $\k\in\mR^n$,
\begin{align}
\frac{\p\sigma_{k+1}}{\p\k_i}(\k)&=\sigma_{k,i}(\k),\\
\sigma_{k+1}(\k)&=\sigma_{k+1,i}(\k)+\k_i\sigma_{k,i}(\k),\\
\sum_{i=1}^n\sigma_{k,i}(\k)&=(n-k)\sigma_k(\k),\\
\sum_{i=1}^n\k_i\sigma_{k,i}(\k)&=(k+1)\sigma_{k+1}(\k),\\
\sum_{i=1}^n\k_i^2\sigma_{k,i}(\k)&=\sigma_1(\k)\sigma_{k+1}(\k)-(k+2)\sigma_{k+2}(\k),\\
\k_1\sigma_{m-1,1}(\k)&\ge \frac{m}{n}\sigma_{m}(\k).
\end{align}
\end{proposition}

\section{The results about flows and evolution equations}
Anisotropic flows of strictly convex hypersurfaces with speed depending on their curvatures, support function and radial function have been considered recently, cf. \cite{IM,DL3,DL4,BIS,CL,BIS3,CHZ,LSW} etc.. These flows usually provide alternative proofs and smooth category approach of the existence of solutions to elliptic PDEs arising in convex body geometry. However, due to the lack of a suitable monotone  functional along this flows, the literature on anisotropic flows of non-convex hypersurfaces is not rich and we didn't find any results for this. To solve this problem, inspired by our previous work \cite{DL3}, we consider a large class of anisotropic flows without global forcing terms.

Let $M_0$ be a closed, smooth and star-shaped hypersurface in $\mathbb{R}^{n+1}$ ($n\geq2$), and $M_0$ encloses the origin. Firstly we describe the homogeneous case as an example.
We shall introduce the following inverse curvature flow and its scaling process,
\begin{equation}\label{x0}
	\begin{cases}
		&\frac{\partial X}{\partial t}(x,t)=\psi(\frac{X}{|X|}) u^{\alpha}\rho^\delta F^{-\beta}(\k_i)\nu(x,t),\\
		&X(\cdot,0)=X_0.
	\end{cases}
\end{equation}
If $\a+\beta+\delta<1$, let $\widetilde X(\cdot,\tau)=\varphi^{-1}(t)X(\cdot,t)$, where
\begin{equation*}
	\tau=\frac{\log((1-\alpha-\delta-\beta)\eta t+C_0)-\log C_0}{(1-\alpha-\delta-\beta)\eta},
\end{equation*}
\begin{equation}\label{x00}
	\varphi(t)=(C_0+(1-\beta-\delta-\alpha)\eta t)^{\frac{1}{1-\beta-\delta-\alpha}}.
\end{equation}
We can set $C_0$ sufficiently large (or small) to make
\begin{equation*}
	\psi\widetilde u^{\alpha-1}\widetilde\rho^\delta\widetilde F^{-\beta}\vert_{M_0}=C_0\psi u^{\alpha-1}\rho^\delta F^{-\beta}\vert_{M_0}>(\text{or }<)\eta.
\end{equation*}
Obviously, $\frac{\p\tau}{\p t}=\varphi^{\a+\delta+\beta-1}$ and $\tau(0)=0$. Then $\widetilde X(\cdot,\tau)$ satisfies the following normalized flow,
\begin{equation}\label{x1}
	\begin{cases}
		\frac{\partial \widetilde X}{\p \tau}(x,\tau)=\psi\widetilde u^{\alpha}\widetilde\rho^\delta\widetilde F^{-\beta}(\widetilde\k)\nu-\eta\widetilde X,\\[3pt]
		\widetilde X(\cdot,0)=\widetilde X_0.
	\end{cases}
\end{equation}
For convenience we still use $t$ instead of $\tau$ to denote the new time variable, and omit the ``tilde'' if no confusions arise. 
Similarly, we can do the same process to a curvature flow. Thus, in this paper, if we replace $\psi(\frac{X}{|X|})u^{\a-1}\rho^\delta$ by $G(X,\nu)$,
for the convenience of calculation, we study the following modified flow
\begin{equation}\label{1.8}
	\begin{cases}
		&\frac{\partial X}{\partial t}(x,t)=\(\Psi\(G(X,\nu)F^{-\beta}(\k_i)\)-\Psi(1)\)X(x,t),\\
		&X(\cdot,0)=X_0,
	\end{cases}
\end{equation}
where $F(x,t)$ is a suitable curvature function of the hypersurface $M_t$ parameterized by $X(\cdot,t): M^n\times[0,T^*)\to \mR^{n+1}$, $\beta>0$, $\k_i$ and $\nu(\cdot,t)$ are the principal curvature and the outer unit normal vector field to $M_t$ respectively. Here $\Psi\in C^\infty(\mR_+)$  is a smooth function satisfying $\Psi'(s)=\frac{d}{ds}\Psi(s)>0$ for all $s>0$. We shall derive the $C^0$ and $C^1$ estimates along flow (\ref{1.8}) and derive the $C^2$-estimate along the following special cases.
\begin{align}
	&\frac{\partial X}{\partial t}(x,t)=\(G(X,\nu)F^{-\beta}(\k_i)-1\)X(x,t),\label{1.9}\\
	&\frac{\partial X}{\partial t}(x,t)=\(1-\widetilde G(X,\nu)F^{\beta}(\k_i)\)X(x,t),\label{1.10}
\end{align}
where $\widetilde G=\frac{1}{G}$.
In fact, the flows (\ref{1.9}) and (\ref{1.10}) are corresponding to the flow (\ref{1.8}) with $\Psi(s)=s$ and $\Psi(s)=-s^{-1}$ respectively.

Flow (\ref{1.8}) is inspired by our previous work \cite{DL3}. However, flow (\ref{1.8}) is more complicated than the flow in \cite{DL3}, since it involves a nonlinear function $G$ and the initial hypersurface is non-convex. Note that $G$ is a function of $X$ other than $\vert X\vert$, which needs more effort to deal with.

In order to prove our main theorems, we will prove the following flow results: Theorems \ref{t3.1} and \ref{t3.2}. And then Theorems \ref{t1.5} and \ref{t1.6} are direct corollaries of Theorems \ref{t3.1} and \ref{t3.2} respectively.
\begin{theorem}\label{t3.1}
Let $F\in C^2(\Gamma)\cap C^0(\p\Gamma)$ satisfy Assumption \ref{a1.3}, and let $X_0(M)$ be the embedding of a closed $n$-dimensional manifold $M^n$ in $\mathbb{R}^{n+1}$ such that $X_0(M)$ is a graph over $\mathbb{S}^n$, and such that $\k\in\Gamma$ for all n-tuples of principal curvatures along $X_0(M)$. Suppose $G^\frac{1}{\beta}$ satisfies
Assumption \ref{a1.1}, one of (\rmnum{2})-(\rmnum{5}) of Assumption \ref{a1.4}, and $uG$ is log-convex with respect to $\nu$ with $\frac{d_\nu G(\nu)}{G}\le \beta$. If $(GF^{-\beta}-1)|_{M_0}$ has a sign and $r_1\le\rho(M_0)\le r_2$,
then the flow (\ref{1.9}) with $\beta>0$ has a unique smooth solution $M_t$ for all time $t>0$. For each $t\in[0,\infty)$, $X(\cdot,t)$ is a parameterization of a smooth, closed, star-shaped, admissible  hypersurface $M_t$ in $\mR^{n+1}$ by $X(\cdot,t)$: $M^n\to \mR^{n+1}$. Moreover,
 a subsequence of $M_t$ converges in $C^\infty$-topology to a positive, smooth admissible solution to $G(X,\nu)F^{-\beta}=1$ along flow (\ref{1.9}).
\end{theorem}
\textbf{Remark:} (1) There have many initial hypersurfaces satisfied $(GF^{-\beta}-1)|_{M_0}\ge0$ or $(GF^{-\beta}-1)|_{M_0}\le0$, such as the sphere with radius $r_1$ or $r_2$ by Assumption \ref{a1.1}.
By Theorem \ref{t3.1}, the equation $F=\breve{G}(X,\nu)$ must have a solution if $\breve{G}=G^\frac{1}{\beta}$. Thus, if $uG$ is log-convex with respect to $\nu$, we have $u\breve{G}^\beta$ is log-convex with respect to $\nu$, i.e. $\log u+\beta\log\breve{G}$ is convex. Since the bounds of $u$ and $\rho$ are based on $r_1$ and $r_2$, we can conclude that $\breve{G}$ is strictly log-convex with respect to $\nu$ by the arbitrary of $\beta$. This is why we need the condition of $G$ in Theorem \ref{t1.5}.

(2) If $G(X,\nu)=\psi(\frac{X}{\rho})u^{\a-1}\rho^\delta$, the theorem follows if $\a\le0$, $\a+\delta+\beta<1$. Thus Theorem \ref{t3.1} generalizes partial results in \cite{DL3}.

(3) We only posed a typical condition in this theorem.
In fact, the condition $uG$ is log-convex can be replaced by one of  the following conditions: (a) $\exists \a<\frac{1}{\beta+1}$, $uG$ is $\a$-convex; (b) $uG$ is strictly $\frac{1}{\beta+1}$-convex.

 If $F=\sigma_{k}^\frac{1}{k}$ and $\Psi(s)=-s^{-1}$, there are some results for special $G=G(\rho,\nu)$, such as \cite{LSW,LSW2,LXZ}. In \cite{LSW2},  Li-Sheng-Wang studied the contracting flows of strictly convex hypersurfaces with $G=\rho^\delta$ and $\beta=1$. In \cite{LXZ}, Li-Xu-Zhang generalized the result in \cite{LSW2} from strictly convex to $k$-convex. We have the following existence result of (\ref{1.1}) by the anisotropic contracting flow (\ref{1.10}), which could be also as an extension of the result in \cite{LXZ}.

\begin{theorem}\label{t3.2}
	Let $F=\sigma_{k}^\frac{1}{k}$, $2\le k\le n-1$, and let $X_0:M^n\rightarrow\mR^{n+1}$ be a smooth, closed, strictly $k$-convex and star-shaped hypersurface in $\mathbb{R}^{n+1}$.
	 Suppose $\wg^{-\frac{1}{\beta}}$ satisfies
	Assumption \ref{a1.1}, one of Assumption \ref{a1.4} and $\frac{d_\nu \wg(\nu)}{\wg}+\beta\ge0$. Further assume either (1) $u\widetilde{G}$ is log-concave with respect to $\nu$ for $\beta\ge k$, or (2) $d_\nu d_\nu(u\wg)-\frac{\beta}{\beta-1}\frac{(d_\nu(u\wg))^2}{u\wg}$ is semi-negative definite for $0<\beta<1$. If $(\wg \sigma_{k}^\frac{\beta}{k}-1)|_{M_0}$ has a sign and $r_1\le\rho(M_0)\le r_2$,
	then the flow (\ref{1.10}) has a unique smooth, strictly $k$-convex solution $M_t$ for all time $t>0$. Moreover,
	a subsequence of $M_t$ converges in $C^\infty$-topology to a positive, smooth, strictly $k$-convex solution to $\wg(X,\nu)\sigma_{k}^\frac{\beta}{k}=1$ along flow (\ref{1.10}).
\end{theorem}

We mention that the longtime existence and convergence of flow $\p_tX=\sF X$ and $\p_tX=u\sF \nu$ are  same. Thus, we have the following corollaries by Theorem \ref{t3.1}, \ref{t3.2} and the scaling process of (\ref{x1}). Corollary \ref{c3.3} generalizes our previous result in \cite{DL3}, and Corollary \ref{c3.4}  generalizes the result in \cite{LXZ}. Moreover, we promote subconvergence to full convergence by Theorem \ref{t6.1}.
\begin{corollary}\label{c3.3}
	Let $F\in C^2(\Gamma_+)\cap C^0(\p\Gamma_+)$ satisfy Assumption \ref{a1.3}, and let $M_0$ be a closed, smooth, star-shaped,  admissible hypersurface in $\mathbb{R}^{n+1}$, $n\ge2$, enclosing the origin. If $\a+\delta+\beta<1$, $\beta>0$, $\a\le0,$
	then flow (\ref{x0}) has a unique smooth and admissible solution $M_t$ for all time $t>0$. After rescaling $X\to \varphi^{-1}(t)X$ defined in (\ref{x00}), the hypersurface $\widetilde M_t=\varphi^{-1}M_t$ converges smoothly to a smooth solution of $\psi u^{\a-1}\rho^\delta F^{-\beta}=\eta$.
\end{corollary}
\begin{corollary}\label{c3.4}
Let $M_0$ be a closed, smooth, star-shaped, strictly $k$-convex hypersurface in $\mathbb{R}^{n+1}$, $n\ge2$, enclosing the origin. If $\a+\delta-\beta>1$, either $\beta\ge k$, $\a\ge0$, or $0<\beta\le1$, $\a=1-\beta$.
Then flow
\begin{equation*}
	\begin{cases}
		&\frac{\partial X}{\partial t}(x,t)=-\psi(\frac{X}{|X|}) u^{\alpha}\rho^\delta \sigma_{k}^\frac{\beta}{k}(\k_i)\nu(x,t),\\
		&X(\cdot,0)=X_0.
	\end{cases}
\end{equation*}
has a unique smooth and strictly $k$-convex solution $M_t$ for all time $t>0$. After a proper rescaling, the hypersurface $\widetilde M_t=\varphi^{-1}M_t$ converges smoothly to a smooth solution of $\psi u^{\a-1}\rho^\delta \sigma_{k}^\frac{\beta}{k}=\eta$.
\end{corollary}

At the end of this section, we derive the following evolution property.
\begin{lemma}\label{l3.3}
	Let $M(t)$ be a smooth family of closed hypersurfaces in $\mR^{n+1}$ evolving along the flow$$\p_tX=\mathscr{F}X,$$
	where $X$ is the position vector and $\mathscr{F}$ is a function defined on $M(t)$. Then we have the following evolution equations.
	\begin{equation}\label{3.1}
		\begin{split}
			\p_tg_{ij}&=2\sF g_{ij}+\rho\nn_i\sF\nn_j\rho+\rho\nn_j\sF\nn_i\rho,\\
			\p_t\nu&=-u\nn\sF,\\
			\p_tu&=u\sF-u\la X,\nn\sF\ra,\\
		\p_th^i_{j}&=-u\nn_j\nn^i\sF-\sF h^i_j-\rho h_{jl}\nn^l\sF\nn^i\rho-\rho h_{jl}\nn^i\sF\nn^l\rho.
		\end{split}
	\end{equation}
\end{lemma}
\begin{proof}
We derive these evolution equations under the normal coordinate system. We denote that $\bar{g}(\cdot,\cdot)=\la\cdot,\cdot\ra$. By direct calculations, we have
\begin{align*}
\frac{\p}{\p t}g_{ij}=&\p_t\la\p_iX,\p_jX\ra\\
=&\la\p_i(\sF X),\p_jX\ra+\la\p_iX,\p_j(\sF X)\ra\\
=&2\sF g_{ij}+\p_i\sF\la X,\p_jX\ra+\p_j\sF\la X,\p_iX\ra\\
=&2\sF g_{ij}+\frac{1}{2}\p_i\sF\p_j\rho^2+\frac{1}{2}\p_j\sF\p_i\rho^2\\
=&2\sF g_{ij}+\rho\nn_i\sF\nn_j\rho+\rho\nn_j\sF\nn_i\rho.
\end{align*}
\begin{equation}\label{3.2}
\begin{split}
\frac{\p}{\p t}\nu=&\la\p_t\nu,\p_iX\ra g^{il}\p_lX\\
=&-\la\nu,\p_i(\sF X)\ra g^{il}\p_lX\\
=&-u\p_i\sF g^{il}\p_lX\\
=&-u\nn\sF.
\end{split}
\end{equation}
Using (\ref{3.2}), we have the evolution of the support function $u$
\begin{align*}
\frac{\p}{\p t}u=&\p_t\la X,\nu\ra\\
=&u\sF-u\la X,\nn\sF\ra.
\end{align*}
Now we calculate the evolution of $h_{ij}$,
\begin{align*}
\frac{\p}{\p t}h_{ij}=&-\p_t\la\p_i\p_j X,\nu\ra\\
=&-\la\p_i\p_j(\sF X),\nu\ra\\
=&-u\nn_j\nn_i\sF+\sF h_{ij}.
\end{align*}
Note that
\begin{align*}
\frac{\p}{\p t}g^{ij}=&-g^{il}(\p_tg_{lm})g^{mj}=-2\sF g^{ij}-\rho\nn^i\sF\nn^j\rho-\rho\nn^j\sF\nn^i\rho.
\end{align*}
Thus,
\begin{align*}
\p_th^i_{j}=&\p_th_{jl}g^{li}+h_{jl}\p_tg^{li}\\
=&-u\nn_j\nn^i\sF-\sF h^i_j-\rho h_{jl}\nn^l\sF\nn^i\rho-\rho h_{jl}\nn^i\sF\nn^l\rho.
\end{align*}
\end{proof}
To derive the $C^1$-estimate, we need the following evolution equation.  For convenience, we denote
$$Q=G(X,\nu)F^{-\beta}.$$
Recall $G(X,\nu)=\phi(X,\nu)\varphi(u)$. we define
\begin{equation*}
	F([a_{ij}])=F(\mu_1,\cdots,\mu_n),
\end{equation*}
where $\mu_1,\cdots,\mu_n$ are the eigenvalues of matrix $[a_{ij}]$. It is not difficult to see that the eigenvalues of $[F^{ij}]=[\frac{\p F}{\p a_{ij}}]$ are $\frac{\p F}{\p\mu_1},\cdots,\frac{\p F}{\p\mu_n}$.
\begin{lemma}\label{l3.4}
Let $\theta=-\log u+\zeta(|X|^2)$. At the spatial maximum point $X(t)$ of $\theta$, if $X$ is not parallel to the normal direction of $X$ at $X(t)$, along flow (\ref{1.8}), we have
\begin{align*}
\cL\theta=&(1-2\zeta'|X|^2)\(\Psi(1)-\Psi(Q)\)+\Psi'Q\(\frac{d_X\phi(X)}{\phi}+\beta\)-\Psi'Q\frac{d_X\phi(u\nu)}{\phi}\\&+2u\la X,e_1\ra\zeta'\Psi'Q\frac{d_\nu \phi(e_1)}{\phi}
+2\zeta'\Psi'Q\la X,e_1\ra^2\frac{u\varphi'}{\varphi}-\beta\frac{u\Psi'Q}{F}F^{ij}(h^2)_{ij}\\&-\beta\frac{u\Psi'Q }{F}\zeta'(2F^{ii}-2uF)-4\beta\frac{u\Psi'Q }{F}((\zeta')^2+\zeta'')\la X,e_1\ra^2 F^{11},
\end{align*}
where $\cL=\frac{\p}{\p t}-\beta\frac{u\Psi'Q F^{ij}}{F}\nn^2_{ij}$ and we choose the local orthonormal frame $\{e_1,\cdots,e_n\}$ on $M$ satisfying $\la X,e_1\ra\ne0$ and $\la X,e_i\ra=0$, $i\ge2$.
\end{lemma}
\begin{proof}
This lemma can be derived by direct calculation.  At the spatial maximum point of $\theta$, we have
$$0=\nn_i\theta=-\frac{\nn_iu}{u}+2\zeta'\la X,e_i\ra.$$
By (\ref{2.13}), we get
$$h_{11}=2u\zeta', \quad h_{1i}=0, \quad i\ge2,$$
where we used $\la X,e_1\ra\ne0$ and $\la X,e_i\ra=0$ for $i\ge 2$.
Therefore, it is possible to rotate the coordinate system such that $\{e_i\}_{i=1}^n$ are the principal curvature directions of the second fundamental form $(h_j^i)$, i.e. $h_j^i=h_{ij}=h_{ij}\delta_j^i$, which means that $(F^{ij})$ is also diagonal. Then we can give the evolution equation of $u$ by (\ref{2.13}),
\begin{equation}\label{3.3}
\begin{split}
\cL u=&\frac{\p u}{\p t}-\beta\frac{u\Psi'Q F^{ij}}{F}\nn^2_{ij}u\\
=&u\(\Psi(Q)-\Psi(1)\)-u\la X,\Psi'\nn Q\ra-\beta\frac{u\Psi'Q}{F}\(\la X,\nn F\ra+F-uF^{ij}(h^2)_{ij}\)\\
=&u\(\Psi(Q)-\Psi(1)\)-u\Psi'Q\la X,e_1\ra\frac{d_X\phi(e_1)}{\phi}-u\Psi'Q\la X,e_1\ra\frac{d_\nu \phi(h_{11}e_1)}{\phi}\\
&-u\Psi'Q\la X,e_1\ra\frac{\varphi'\nn_1u}{\varphi}-\beta u\Psi'Q+\beta\frac{u^2\Psi'Q}{F}F^{ij}(h^2)_{ij}\\
=&u\(\Psi(Q)-\Psi(1)\)-u\Psi'Q\la X,e_1\ra\frac{d_X\phi(e_1)}{\phi}-2u^2\zeta'\Psi'Q\la X,e_1\ra\frac{d_\nu \phi(e_1)}{\phi}\\
&-2u\zeta'\Psi'Q\la X,e_1\ra^2\frac{u\varphi'}{\varphi}-\beta u\Psi'Q+\beta\frac{u^2\Psi'Q}{F}F^{ij}(h^2)_{ij}.
\end{split}
\end{equation}
Note that $X=\la X,e_1\ra e_1+u\nu$.
Combing (\ref{1.8}) and (\ref{3.3}), we have
\begin{align*}
\cL\theta=&-\frac{\cL u}{u}+2\zeta'(\Psi(Q)-\Psi(1))|X|^2\\
&-\beta\frac{u\Psi'Q F^{ii}}{F}\(4((\zeta')^2+\zeta'')\la X,e_1\ra^2\delta_{1i}+\zeta'(2-2uh_{ii})\)\\
=&(1-2\zeta'|X|^2)\(\Psi(1)-\Psi(Q)\)+\Psi'Q\(\frac{d_X\phi(X)}{\phi}+\beta\)-\Psi'Q\frac{d_X\phi(u\nu)}{\phi}\\&+2u\la X,e_1\ra\zeta'\Psi'Q\frac{d_\nu \phi(e_1)}{\phi}
+2\zeta'\Psi'Q\la X,e_1\ra^2\frac{u\varphi'}{\varphi}-\beta\frac{u\Psi'Q}{F}F^{ij}(h^2)_{ij}\\&-\beta\frac{u\Psi'Q }{F}\zeta'(2F^{ii}-2uF)-4\beta\frac{u\Psi'Q }{F}((\zeta')^2+\zeta'')\la X,e_1\ra^2 F^{11}.
\end{align*}
\end{proof}

\section{The $C^0$ and $C^1$ estimates}
In this section, we will establish the $C^0$ and $C^1$ estimates of general flow (\ref{1.8}) without the concavity of $F$ and boundary condition $F\vert_{\p\Gamma}=0$ since this two conditions are not needed to derive the $C^0$ and $C^1$ estimates. In the rest of this section we still say that $F$ satisfies Assumption \ref{a1.3} and will not repeat that $F$ doesn't need to satisfy   concavity  and $F\vert_{\p\Gamma}=0$.

We consider the flow equation (\ref{1.8}) of radial graphs over $\mathbb{S}^n$ in $\mathbb{R}^{n+1}$. It is easy to see that the evolution of the radial function $\rho=\rho(X(z,t),t)$ satisfies the following parabolic initial value problem on $\mathbb{S}^n$,
\begin{equation}\label{2.17}
	\begin{cases}
		\p_t\rho&=\(\Psi\(G(X,\nu)F^{-\beta}(\k_i)\)-\Psi(1)\)\rho, \;\;(z,t)\in\mathbb{S}^n\times [0,\infty),\\
		\rho(\cdot,0)&=\rho_0,
	\end{cases}
\end{equation}
where $\rho_0$ is the radial function of the initial hypersurface.

Equivalently, the equation for $\gamma$ satisfies
\begin{equation}\label{2.18}
	\p_t\gamma=\Psi\(G(X,\nu)F^{-\beta}(\k_i)\)-\Psi(1).
\end{equation}

We first show the $C^0$-estiamte of the solution to (\ref{2.17}).
\begin{lemma}\label{l4.2}
Let $\rho(x,t)$, $t\in[0,T)$, be a smooth, star-shaped solution to (\ref{2.17}). If $G^\frac{1}{\beta}$ satisfies (\ref{1.2}), then
\begin{equation*}
r_1\leq \rho(\cdot,t)\leq r_2.
\end{equation*}
\end{lemma}
\begin{proof}
Let $\rho_{\max}(t)=\max_{z\in \mS^n}\rho(\cdot,t)=\rho(z_t,t)$. For fixed time $t$, at the point $z_t$, we have $$D_i\rho=0 \text{ and } D^2_{ij}\rho\leq0.$$
Note that $\omega=1$, $u=\frac{\rho}{\omega}=\rho$ and
\begin{equation}\label{4.1}
	\begin{split}
h^i_j &=\frac{1}{\rho^2\sqrt{\rho^2+|D\rho|^2}}(e^{ik}-\frac{\rho^i\rho^k}{\rho^2+|D\rho|^2})(-\rho D_kD_j\rho+2\rho_k\rho_j+\rho^2e_{kj})\\
&=-\rho^{-2}\rho_{ij}+\rho^{-1}e_{ij}.
\end{split}
\end{equation}
At the point $z_t$,
 we have $F^{-\beta}(h^i_j)\leq\eta(\frac{1}{\rho})^{-\beta}$.
 Thus
$$\frac{d}{dt}\rho_{\max}\leq\rho\(\Psi\( G(X,\frac{X}{\rho})\eta(\frac{1}{\rho})^{-\beta}\)-\Psi(1)\),$$
where we have used $\Psi$ is increasing and $\nu=\frac{X}{\rho}$ at $z_t$. Since $r_1\le\rho(M_0)\le r_2$, we can consider the point where the maximum point touches the sphere $\vert X\vert=r_2$ for the first time. Since $G^\frac{1}{\beta}$ satisfies (\ref{1.2}), we have $\p_t\rho\le0$ at this point, i.e.
$\rho_{\max}\le r_2.$
Similarly, we can derive $\rho_{\min}\ge r_1.$
\end{proof}

\begin{lemma}\label{l4.3}
 Let $X(\cdot,t)$ be the solution to the flow (\ref{1.8}) which encloses the origin for $t\in[0,T)$. If $Q-1$ has a sign at $t=0$, then it has a sign for all time. Moreover, if $G$ satisfies $\frac{\p}{\p\rho}\(\rho G^\frac{1}{\beta}(X,\nu)\)\le0$, there is a positive constant $C_1$ depending on the initial hypersurface and $\beta$,  such that
$$\frac{1}{C_1}\leq Q\leq C_1.$$
\end{lemma}
\begin{proof}
In order to calculate the evolution equation of $Q$, we need to deduce the evolution equation of $F$ firstly by Lemma \ref{l3.3}.
\begin{equation}\label{4.2}
	\begin{split}
		\frac{\p F}{\p t}=&F^{ij}\frac{\p h^i_j}{\p t}\\
		=&F^{ij}\(-u\nn^i\nn_j(\Psi(Q))-(\Psi(Q)-\Psi(1))h^i_j-2\rho h_{jl}\nn^l(\Psi(Q))\nn^i\rho\)\\
		=&F^{ij}\(-u\Psi''\nn^iQ\nn_jQ-u\Psi'\nn^i\nn_jQ-(\Psi(Q)-\Psi(1))h^i_j-2\rho h_{jl}\Psi'\nn^lQ\nn^i\rho\).
			\end{split}
\end{equation}
At the point where $Q$ attains its spatial maximum or minimum, $\nn Q=0$. We use (\ref{1.8}), Lemma \ref{l3.3} and (\ref{4.2}) to deduce
\begin{equation}\label{4.3}
\begin{split}
\p_tQ=&\p_t(G(X,\nu)F^{-\beta})\\
=&Q\frac{d_XG(X)(\Psi(Q)-\Psi(1))}{G}-Q\frac{d_\nu G(u\Psi'\nn Q)}{G}-\beta\frac{Q}{F}\p_tF\\
=&Q\(\frac{d_XG(X)}{G}+\beta\)(\Psi(Q)-\Psi(1))-Q\frac{d_\nu G(u\Psi'\nn Q)}{G}\\
&+\beta\frac{Q}{F}F^{ij}\(u\Psi''\nn^iQ\nn_jQ+u\Psi'\nn^i\nn_jQ+2\rho h_{il}\Psi'\nn^lQ\nn^j\rho\).
\end{split}
\end{equation}
Note that $\Psi$ is increasing. By the maximum principle, if $Q-1$ has a sign at $t=0$, then it has a sign for all time.

If $\frac{\p}{\p\rho}\(\rho G^\frac{1}{\beta}(X,\nu)\)\le0$, we have $\frac{\p}{\p\rho}\(\rho^\beta G(X,\nu)\)\le0$. Thus
\begin{align*}
0\ge\frac{\p}{\p\rho}\(\rho^\beta G(X,\nu)\)=\rho^{\beta-1}\(\beta G+\rho\frac{\p G(X,\nu)}{\p\rho}\)=&\rho^{\beta-1}\(\beta G+\rho d_XG(\frac{\p X}{\p\rho})\)\\
=&\rho^{\beta-1}\(\beta G+d_XG(X)\),
\end{align*}
i.e.
\begin{equation*}
\frac{d_XG(X)}{G}+\beta\le0.
\end{equation*}
Then applying the maximum principle to (\ref{4.3}) we know that $\frac{1}{C_1}\leq Q\leq C_1$, where $C_1$ is a positive constant depending on the initial hypersurface and $\beta$.
\end{proof}

According to Lemma \ref{l3.4} and \ref{l4.3}, we can derive the $C^1$ estimates.
\begin{lemma}\label{l4.4}
Let $X(\cdot,t)$ be the solution to the flow (\ref{1.8}) which encloses the origin for $t\in[0,T)$. If $F$ satisfies Assumption \ref{a1.3}, $G^\frac{1}{\beta}$ satisfies (\rmnum{2}) or (\rmnum{4}) or (\rmnum{5}) of Assumption \ref{a1.4}, and $(Q-1)|_{M_0}$ has a sign and $\rho$ has positive lower and upper bound,
then there are positive constants $C_2$, $C_3$ depending on the initial hypersurface, $\Vert G\Vert_{C^0}$, $\Vert G\Vert_{C^1}$ and $\beta$,  such that
	$$\vert D \rho\vert\leq C_2\quad \text{  and } \quad u\ge C_3.$$
\end{lemma}
\begin{proof}
We only need to obtain a positive lower bound of $u$. Recall Lemma \ref{l3.4}. Assume $X_0$ is the maximum value point of $\theta$. If $X$ is parallel to the normal direction of $X$ at $X_0$, $u(X_0)=\la X_0,\nu\ra=\rho(X_0)$. Then we have the lower bound of $u$ by the lower bound of $\rho$. Thus it suffice to derive the lower bound of $u$ at the case that $X$ is not parallel to the normal direction of $X$ at $X_0$. We choose the same coordinate system as in Lemma \ref{l3.4}. Then we have
\begin{equation}\label{4.4}
\begin{split}
\cL\theta=&(1-2\zeta'|X|^2)\(\Psi(1)-\Psi(Q)\)+\Psi'Q\(\frac{d_X\phi(X)}{\phi}+\beta\)-\Psi'Q\frac{d_X\phi(u\nu)}{\phi}\\&+2u\la X,e_1\ra\zeta'\Psi'Q\frac{d_\nu \phi(e_1)}{\phi}
+2\zeta'\Psi'Q\la X,e_1\ra^2\frac{u\varphi'}{\varphi}-\beta\frac{u\Psi'Q}{F}F^{ij}(h^2)_{ij}\\&-\beta\frac{u\Psi'Q }{F}\zeta'(2F^{ii}-2uF)-4\beta\frac{u\Psi'Q }{F}((\zeta')^2+\zeta'')\la X,e_1\ra^2 F^{11}.
\end{split}
\end{equation}
We divide this proof into three cases.

\noindent(1) The case that $G^\frac{1}{\beta}$ satisfies (\rmnum{2}) of Assumption \ref{a1.4}. Choose $$\zeta(t)=-\frac{\a}{t}$$
for sufficiently large positive constant $\a$.
We choose $(Q-1)|_{M_0}\le0$ at this case. Then $\Psi(1)-\Psi(Q)\ge0$ by Lemma \ref{l4.3} and the monotonicity of $\Psi$. We let $1-2\zeta'|X|^2\le0$ for sufficiently large positive constant $\a$. Using (\rmnum{2}) of Assumption \ref{a1.4}, we get
\begin{equation*}
	\begin{split}
\frac{\cL\theta}{\Psi'Q}\le&\frac{d_X\phi(X)}{\phi}+\beta-\eps\a+u\(-\frac{d_X\phi(\nu)}{\phi}+C_3\a\frac{d_\nu \phi(e_1)}{\phi}+C_4\a u\)\\
&-\dfrac{u}{F}(C_5\a^2-C_6\a)F^{11}.
	\end{split}
\end{equation*}
By the assumption on the $C^0$ bound, we have $\frac{d_X\phi(X)}{\phi},\frac{d_X\phi(\nu)}{\phi},\frac{d_\nu \phi(e_1)}{\phi}\le C$. For sufficiently large $\a$ and small $u$, we get
\begin{equation*}
\cL\theta\le0.
\end{equation*}
That is, at $X_0$, $u$ has lower bound.

\noindent(2) The case that $G^\frac{1}{\beta}$ satisfies (\rmnum{4}) of Assumption \ref{a1.4}. By the assumption on the $C^0$ bound, (\ref{4.3}) and (\rmnum{4}) of Assumption \ref{a1.4}, we get
\begin{equation*}
	\frac{d_XG(X)}{G}+\beta=\frac{d_X\phi(X)}{\phi}+\beta\le-\eps.
\end{equation*}
At this case, through different choices of $\zeta$, we can derive the lower bound of $u$ by choosing $(Q-1)|_{M_0}\le0$ or $(Q-1)|_{M_0}\ge0$. We only give a proof for one of the cases. We choose $(Q-1)|_{M_0}\ge0$ and $\zeta=1$. Then
rewriting (\ref{4.4}),
\begin{align*}
\frac{\cL\theta}{\Psi'Q}\le-\eps-u\frac{d_X\phi(\nu)}{\phi}\le0
\end{align*}
if $u$ is sufficiently small. Then we get the lower bound of $u$.

\noindent(3) The case that $G^\frac{1}{\beta}$ satisfies (\rmnum{5}) of Assumption \ref{a1.4}. 
We choose $$\zeta(t)=-\frac{\a}{t}$$
for sufficiently large positive constant $\a$ and choose $(Q-1)|_{M_0}\le0$ at this case. By (\ref{4.3}) we get
\begin{equation*}
	\begin{split}
\frac{\cL\theta}{\Psi'Q}&\le u\(-\frac{d_X\phi(\nu)}{\phi}-C_7\a\frac{\sum_{i}F^{ii}}{F}-\beta\frac{F^{ii}h_{ii}^2}{F}+C_4\a u\)-\dfrac{u}{F}(C_5\a^2-C_6\a)F^{11}\\
&\le u\(-\frac{d_X\phi(\nu)}{\phi}-C_8\sqrt{\a}+C_4\a u\)-\dfrac{u}{F}(C_5\a^2-C_6\a)F^{11}\le0,
	\end{split}
\end{equation*}
where we used $C_7\a F^{ii}+\beta F^{ii}h_{ii}^2\ge C_8\sqrt{\a} F$. Then we get the lower bound of $u$.
\end{proof}
We would like to get the upper bound of $|D\gamma|$ to show that $u$ has a lower bound if $G$ satisfies (\rmnum{3}) of Assumption \ref{a1.4}.

\begin{lemma}\label{l4.5}
Let $\rho(x,t)$, $t\in[0,T)$, be a smooth, star-shaped solution to (\ref{2.17}). If $F$ satisfies Assumption \ref{a1.3}, $G^\frac{1}{\beta}$ satisfies (\rmnum{3}) of Assumption \ref{a1.4} and $\rho$ has positive lower and upper bound, then there are positive constants $C_9$, $C_{10}$ depending on the initial hypersurface, $\Vert G\Vert_{C^0}$, $\Vert G\Vert_{C^1}$ and $\beta$,  such that
$$\vert D \rho\vert\leq C_9\quad \text{  and } \quad u\ge C_{10}.$$
\end{lemma}
\begin{proof}
Consider the auxiliary function $O=\frac{1}{2}\vert D\g\vert^2$. At the point where $O$ attains its spatial maximum, we have
\begin{gather*}
D\omega=0,\\
0=D_iO=\sum_{l}\g_{li}\g_{l},\\
0\geq D_{ij}^2O=\sum_{l}\g_{li}\g_{lj}+\sum_{l}\g_{l}\g_{lij}.
\end{gather*}
By (\ref{2.14}) and (\ref{2.18}), we deduce
\begin{equation}\label{4.5}
	\begin{split}
\p_t\gamma=&\Psi(G(X,\nu) F^{-\beta})-\Psi(1)\\
=&\Psi \(G(\rho\frac{\p}{\p\rho},\frac{1}{\om}(\frac{\p}{\p\rho}-\phi^{-2}\rho_i\frac{\p}{\p x_i}))F^{-\beta}(\frac{1}{\rho\omega}(\delta_{ij}-(e^{ik}-\frac{\gamma_i\gamma_k}{\omega^2})\gamma_{kj}))\)-\Psi(1)\\
=&\Psi \(\rho^\beta \om^\beta \phi(\rho\frac{\p}{\p\rho},\frac{1}{\om}(\frac{\p}{\p\rho}-\rho^{-2}\rho_i\frac{\p}{\p x_i}))\varphi(\frac{\rho}{\om}) F^{-\beta}\)-\Psi(1).
	\end{split}
\end{equation}
We remark that here $F^{-\beta}=F^{-\beta}([\delta_{ij}-(e^{ik}-\frac{\gamma_i\gamma_k}{\omega^2})\gamma_{kj}])$ and $$F^{ij}=F^{ij}([\delta_{ij}-(e^{ik}-\frac{\gamma_i\gamma_k}{\omega^2})\gamma_{kj}]).$$
We denote $\p_0=\frac{\p}{\p\rho}$. By $ds^2=d\rho^2+\rho^2e$, we can derive
\begin{align*}
\bar\Gamma^k_{i0}\p_k=\rho^{-1}\p_i, \quad\bar\Gamma^0_{i0}=0,\quad\bar\Gamma^k_{il}\p_k=\g_i\p_l+\g_l\p_i-\delta_{il}\g_k\p_k.
\end{align*}
With the above preparations, we can get the evolution equation of $O$.
\begin{align}\label{4.6}
\p_tO_{\max}=&\sum\g^l\g_{tl}\notag\\
=&Q\Psi'\(\beta|D\g|^2+\frac{d_X\phi(X|D\g|^2+\g^l\p_l)}{\phi}+\frac{d_\nu\phi(\rho^{-1}\g^l\p_l+|D\g|^2\p_0)}{\phi\om}\notag\\
&\frac{u\varphi_u}{\varphi}|D\g|^2+\frac{\beta F^{ij}}{F}\g^l\g_{kjl}(e^{ik}-\frac{\gamma_i\gamma_k}{\omega^2})\)\notag\\
=&Q\Psi'\((\frac{d_X\phi(X)}{\phi}+\frac{u\varphi_u}{\varphi}+\beta)|D\g|^2+\frac{d_X\phi(\p_l)}{\phi}\g^l+\frac{d_\nu\phi(\rho^{-1}\p_l)}{\phi}\frac{\g^l}{\om}\notag\\
&+\frac{d_\nu\phi(\p_0)}{\phi}\frac{|D\g|^2}{\om}+\frac{\beta F^{ij}}{F}\g^l\g_{kjl}(e^{ik}-\frac{\gamma_i\gamma_k}{\omega^2})\),
\end{align}
where we have used $\sum_{l}\g_{li}\g_{l}=0$. By the Ricci identity,$$D_l\g_{ij}=D_j\g_{li}+\e_{il}\g_j-\e_{ij}\g_l,$$
we get
\begin{align}\label{4.7}
F^{ij}\g^l\g_{kjl}(e^{ik}-\frac{\gamma_i\gamma_k}{\omega^2})=&F^{ij}\g^l(\g_{lkj}+\g_j\e_{lk}-\g_l\e_{kj})(e^{ik}-\frac{\gamma_i\gamma_k}{\omega^2})\notag\\
\leq&F^{ij}(-\g_{lk}\g_{lj}+\g_l\g_je_{lk}-\vert D\g\vert^2e_{kj})(e^{ik}-\frac{\gamma_i\gamma_k}{\omega^2})\notag\\
\leq&F^{ij}(-\g_{li}\g_{lj}+\g_i\g_j-\vert D\g\vert^2\delta_{ij}).
\end{align}
By a direct calculation, that $G^\frac{1}{\beta}$ satisfies (\rmnum{3}) of Assumption \ref{a1.4} is equivalent to
$$\frac{d_X\phi(X)}{\phi}+\frac{u\varphi_u}{\varphi}+\beta\le-\eps'.$$
According to (\ref{4.6}) and (\ref{4.7}), we can derive
\begin{align}\label{4.8}
\p_tO_{\max}\le&Q\Psi'\(-\eps'|D\g|^2+\frac{d_X\phi(\p_l)}{\phi}\g^l+\frac{d_\nu\phi(\rho^{-1}\p_l)}{\phi}\frac{\g^l}{\om}\notag\\
&+\frac{d_\nu\phi(\p_0)}{\phi}\frac{|D\g|^2}{\om}+\frac{\beta F^{ij}}{F}(-\g_{li}\g_{lj}+\g_i\g_j-\vert D\g\vert^2\delta_{ij})\).
\end{align}
In terms of the positive definite of the symmetric matrix $[F^{ij}]$, we have $F^{ij}\g_{li}\g_{lj}\geq0$ and $F^{ij}\g_i\g_j\leq\max_iF^{ii}\vert D\g\vert^2$. If $|D\g|^2$ is sufficiently large, by the bound of $\rho$ and $\phi$, we have
$$\p_tO_{\max}\leq0.$$
Then we have $\vert D\g(\cdot,t)\vert\leq C_9$ for a positive constant $C_9$.
\end{proof}
If $F=\sigma_{k}^\frac{1}{k}$, we can derive the $C^1$-estimate by weaker conditions.
\begin{lemma}\label{l4.6}
Let $X(\cdot,t)$ be the solution to the flow (\ref{1.8}) which encloses the origin for $t\in[0,T)$. If $F=\sigma_{k}^\frac{1}{k}$, $G^\frac{1}{\beta}$ satisfies one of Assumption \ref{a1.4}, $(Q-1)|_{M_0}$ has a sign and $\rho$ has positive lower and upper bound,
then there are positive constants $C_{11}$, $C_{12}$ depending on the initial hypersurface,  $\Vert G\Vert_{C^0}$, $\Vert G\Vert_{C^1}$ and $\beta$,  such that
	$$\vert D \rho\vert\leq C_{11}\quad \text{  and } \quad u\ge C_{12}.$$
\end{lemma}
\begin{proof}
It suffice to derive a positive lower bound of $u$. Similar to Lemma \ref{l4.4}, by (\ref{4.4}), we have
\begin{equation}\label{4.9}
	\begin{split}
		\cL\theta=&(1-2\zeta'|X|^2)\(\Psi(1)-\Psi(Q)\)+\Psi'Q\(\frac{d_X\phi(X)}{\phi}+\beta\)-\Psi'Q\frac{d_X\phi(u\nu)}{\phi}\\
		&+2u\la X,e_1\ra\zeta'\Psi'Q\frac{d_\nu \phi(e_1)}{\phi}
		+2\zeta'\Psi'Q\la X,e_1\ra^2\frac{u\varphi'}{\varphi}-\frac{\beta u\Psi'Q}{k\sigma_{k}}\sigma_{k}^{ij}(h^2)_{ij}\\&-\frac{2\beta u\Psi'Q}{k\sigma_{k}}\zeta'\sigma_{k}^{ii}+2\beta u^2\Psi'Q-\frac{4\beta u\Psi'Q}{k\sigma_{k}}((\zeta')^2+\zeta'')\la X,e_1\ra^2\sigma_{k}^{11}.
	\end{split}
\end{equation}
If $G^\frac{1}{\beta}$ satisfies (\rmnum{2})-(\rmnum{5}) of Assumption \ref{a1.4}, we have been proved in the above lemmas. Thus
we divide this proof into two cases.

\noindent(1) The case that $G^\frac{1}{\beta}$ satisfies (\rmnum{1}) of Assumption \ref{a1.4}.  Choose $$\zeta(t)=\frac{\a}{t}$$
for sufficiently large positive constant $\a$.
We choose $(Q-1)|_{M_0}\ge0$ at this case. Then $\Psi(1)-\Psi(Q)\le0$ by Lemma \ref{l4.3} and the monotonicity of $\Psi$. Recall $h_{11}=2u\zeta'\le0$ by the choice of function $\zeta$. $\phi(X,\nu)$ has uniform bounds since $\phi(X,\nu)$ is defined in a compact set by $C^0$ estimates.
Note that $G^\frac{1}{\beta}$ satisfies (\rmnum{1}) of Assumption \ref{a1.4}. By Lemma \ref{l4.3}, we derive the bounds of $Q$,
i.e. we derive the bounds of $\sigma_{k}$. Thus the monotonicity of $\sigma_{k-1}$ and  the Newton-Maclaurin inequality (Lemma \ref{l2.5}) imply
$$\sigma_k^{11}=\sigma_{k-1,i}\ge\sigma_{k-1}\ge\frac{k}{(n-k+1)(k-1)}(C_n^k)^\frac{1}{k}\sigma_{k}^\frac{k-1}{k}\ge C.$$
Therefore by the above inequality, Proposition \ref{p2.6}, (\rmnum{1}) of Assumption \ref{a1.4} and (\ref{4.9}), we have
\begin{align*}
\cL\theta\le&-\Psi'Q\frac{d_X\phi(u\nu)}{\phi}-2u\la X,e_1\ra\frac{\a}{\rho^4}\Psi'Q\frac{d_\nu \phi(e_1)}{\phi}
+\frac{2\beta u\Psi'Q}{k\sigma_{k}}\frac{\a}{\rho^4}(n-k+1)\sigma_{k-1}\\&+2\beta u^2\Psi'Q-\frac{4\beta u\Psi'Q}{k\sigma_{k}}(\frac{\a^2}{\rho^8}+\frac{2\a}{\rho^6})\la X,e_1\ra^2\sigma_{k}^{11}\\
\le&u\Psi'Q\(C_{13}+C_{14}\a-C_{15}\a^2\).
\end{align*}
Thus, for sufficiently large $\a$,
at $X_0$, $u$ has lower bound.

\noindent(2) The case that $G^\frac{1}{\beta}$ satisfies (\rmnum{6}) of Assumption \ref{a1.4}. We only need to prove the case $\frac{u\varphi_u}{\varphi}\ge\eps$ since $\varphi$ is smooth and the case $\frac{u\varphi_u}{\varphi}\le-\eps$ has been proved in Lemma \ref{l4.4}. Choose $$\zeta(t)=\frac{\a}{t}$$
for sufficiently large positive constant $\a$.
We choose $(Q-1)|_{M_0}\ge0$ at this case. Then $\Psi(1)-\Psi(Q)\le0$ by Lemma \ref{l4.3} and the monotonicity of $\Psi$. Recall $h_{11}=2u\zeta'\le0$ by the choice of function $\zeta$. Thus by the monotonicity of $\sigma_{k}$, we have
$$\sigma_k^{11}\ge\sigma_{k-1}.$$
Note that we don't have the bounds of $\sigma_{k}$ at this case.
Therefore by the above inequality, Proposition \ref{p2.6}, (\rmnum{1}) of Assumption \ref{a1.4} and (\ref{4.9}), we have
\begin{align*}
\cL\theta\le&\Psi'Q\(\frac{d_X\phi(X)}{\phi}+\beta\)-\Psi'Q\frac{d_X\phi(u\nu)}{\phi}-2u\la X,e_1\ra\frac{\a}{\rho^4}\Psi'Q\frac{d_\nu \phi(e_1)}{\phi}
	\\&-2\frac{\a}{\rho^4}\Psi'Q\la X,e_1\ra^2\frac{u\varphi'}{\varphi}+\frac{2\beta u\Psi'Q}{k\sigma_{k}}\frac{\a}{\rho^4}(n-k+1)\sigma_{k-1}+2\beta u^2\Psi'Q\\&-\frac{4\beta u\Psi'Q}{k\sigma_{k}}(\frac{\a^2}{\rho^8}+\frac{2\a}{\rho^6})\la X,e_1\ra^2\sigma_{k-1}\\
\le&\Psi'Q\(C_{16}+(-\eps C_{17}+C_{18}u)\a+C_{19}u+\frac{\sigma_{k-1}}{\sigma_{k}}(-C_{20}\a^2+C_{21}\a)\).
\end{align*}
Thus, for sufficiently large $\a$ and sufficiently small $u$,
at $X_0$, we can derive the lower bound of $u$.
\end{proof}
At the end of this section, we shall derive the bounds of $F$. If $G^\frac{1}{\beta}$ satisfies (\rmnum{1}), or (\rmnum{3})-(\rmnum{5}) of Assumption \ref{a1.4}, we have got the bounds of $F$ by Lemma \ref{l4.3}. 
Here, we give the bound of $F$ in some more general cases.
\begin{lemma}\label{l4.7}
Let $X(\cdot,t)$ be the solution to the flow (\ref{1.9}) which encloses the origin for $t\in[0,T)$. If $F$ satisfies Assumption \ref{a1.3}, $(Q-1)|_{M_0}\le0$ and $\rho,u$ have positive lower and upper bound,
then there is a positive constant $C_{22}$ depending on the initial hypersurface,  $\Vert G\Vert_{C^0}$, $\Vert G\Vert_{C^1}$ and $\beta$,  such that
$$\frac{1}{C_{22}}\leq F\leq C_{22}.$$
\end{lemma}
\begin{proof}
Note that $Q=GF^{-\beta}$. According to $(Q-1)|_{M_0}\le0$, the bound of $u$ and Lemma \ref{l4.3}, we can derive the lower bound of $F$. Next we shall derive the upper bound of $F$.  Let $\bar G=uG$, $P=uGF^{-\beta}=\bar GF^{-\beta}$, $\cL=\p_t-\beta\frac{P}{F}F^{ij}\nn_{ij}^2$, $\theta=\log P-\frac{A}{2}\rho^2$, where $A>0$ is a sufficiently large constant. Thus it suffice to get the lower bound of $P$ or $\theta$. At the spatial minimum point of $\theta$, we choose the normal coordinate system such that
$$h_j^i=h_{ij}=h_{ii}\delta_{ij},\qquad g_{ij}=\delta_{ij}.$$
Then $\nn_iu=h_i^l\nn_l\Phi=h_{ii}\nn_i\Phi$. Meanwhile, we have
$$0=\nn\theta=\frac{\nn P}{P}-A\rho\nn\rho.$$
Thus
\begin{equation}\label{4.10}
\begin{split}
\cL P=&\frac{P}{\bar G}\(d_X\bar G(X)(\frac{P}{u}-1)-d_\nu\bg(u\nn\frac{P}{u})\)\\
&+\beta\frac{P}{F}F^{ii}\(u\nn^2_{ii}\frac{P}{u}+(\frac{P}{u}-1)h_{ii}+2\nn_iu\nn_i\frac{P}{u}\)-\beta\frac{P}{F}F^{ij}\nn_{ij}^2P\\
=&(\frac{P}{u}-1)P(\frac{d_X\bar G(X)}{\bg}+\beta)-\frac{P}{\bar G}d_\nu\bg(\nn P-\frac{P\nn u}{u})-\beta\frac{P^2}{F}F^{ii}\frac{\nn^2_{ii}u}{u}\\
=&(\frac{P}{u}-1)P(\frac{d_X\bar G(X)}{\bg}+\beta)-\frac{P}{\bar G}d_\nu\bg(\nn P-\frac{P\nn u}{u})\\
&-\beta\frac{P^2}{F}(\frac{\nn_lF\nn^l\Phi}{u}+\frac{F}{u}-F^{ii}h_{ii}^2)\\
=&(\frac{P}{u}-1)P(\frac{d_X\bar G(X)}{\bg}+\beta)-\frac{P}{\bar G}d_\nu\bg(\nn P-\frac{P\nn u}{u})-\beta\frac{P^2}{u}+\beta\frac{P^2}{F}F^{ii}h_{ii}^2\\
&+P^2\frac{\nn_i\Phi}{u}\(\frac{\nn_iP}{P}-\frac{d_X\bg(e_i)}{\bg}-\frac{d_\nu\bg(h_{ii}e_i)}{\bg}\)\\
=&-P(\frac{d_X\bar G(X)}{\bg}+\beta)+P^2\frac{d_X\bg(\nu)}{\bg}+(\frac{\nn_i\Phi}{u}-\frac{d_\nu\bg(e_i)}{\bg})P\nn_iP+\beta\frac{P^2}{F}F^{ii}h_{ii}^2.
\end{split}
\end{equation}
\begin{equation}\label{4.11}
\begin{split}
\cL(\frac{A}{2}\rho^2)=&A\rho^2(\frac{P}{u}-1)-A\frac{\beta P}{F}F^{ij}(\nn_i\rho\nn_j\rho+\rho\nn_{ij}^2\rho)\\
=&A\rho^2(\frac{P}{u}-1)-A\frac{\beta P}{F}F^{ij}(g_{ij}-uh_{ij}).
\end{split}
\end{equation}
Combining (\ref{4.10}) and (\ref{4.11}), we derive
\begin{align*}
\cL\theta=&\frac{\cL P}{P}+\beta\frac{F^{ij}}{F}\frac{\nn_iP\nn_jP}{P}-\cL(\frac{A}{2}\rho^2)\\
\ge&(\frac{A}{2}\rho^2-C_{23}P-C_{24})+A(\frac{1}{2}\rho^2-C_{25}P)>0
\end{align*}
for sufficiently small $P$ and sufficiently large $A$. Thus by the maximum principle, we derive the lower bound of $P$.
\end{proof}
\begin{lemma}\label{l4.8}
	Let $X(\cdot,t)$ be the solution to the flow (\ref{1.10}) which encloses the origin for $t\in[0,T)$. If $F$ satisfies Assumption \ref{a1.3}, and $\rho,u$ have positive lower and upper bound,
	then there is a positive constant $C_{26}$ depending on the initial hypersurface,  $\Vert\widetilde G\Vert_{C^0}$, $\Vert\widetilde G\Vert_{C^1}$ and $\beta$,  such that
	$$F\geq C_{26}.$$
\end{lemma}
\begin{proof}
Let $Q=\wg F^\beta$, $\cL=\p_t-\beta\frac{uQ}{F}F^{ij}\nn_{ij}^2$, $\theta=\log Q+Au$. Similar to Lemma \ref{l4.7},  it suffice to get the lower bound of $Q$ or $\theta$. At the spatial minimum point of $\theta$, we choose the normal coordinate system such that
$$h_j^i=h_{ij}=h_{ii}\delta_{ij},\qquad g_{ij}=\delta_{ij}.$$
Then $\nn_iu=h_i^l\nn_l\Phi=h_{ii}\nn_i\Phi$. Meanwhile, we have
$$0=\nn\theta=\frac{\nn Q}{Q}+A\nn u.$$
We can derive
\begin{equation}\label{4.12}
\begin{split}
\cL Q=&Q\(\frac{d_X\wg(X)(1-Q)}{\wg}+\frac{d_\nu\wg(u\nn Q)}{\wg}\)
+\beta\frac{QF^{ii}}{F}\((Q-1)h_{ii}+2\nn_iu\nn_iQ\)\\
=&Q\((\frac{d_X\wg(X)}{\wg}-\beta)(1-Q)+\frac{ud_\nu\wg(\nn Q)}{\wg}+2\beta\frac{F^{ii}}{F}\nn_iu\nn_iQ\).
\end{split}
\end{equation}
\begin{equation}\label{4.13}
\begin{split}
	\cL u=&u(1-Q)+u\la X,\nn Q\ra-\beta\frac{uQ}{F}F^{ij}\nn_{ij}^2u\\
=&u-(1+\beta)uQ+uQ\frac{d_X\wg(\nn\Phi)}{\wg}+uQ\frac{d_\nu\wg(\nn u)}{\wg}+\beta u^2Q\frac{F^{ij}(h^2)_{ij}}{F}.
\end{split}
\end{equation}
Thus, by (\ref{4.12}) and (\ref{4.13}),
\begin{align*}
\cL\theta=&\frac{\cL Q}{Q}+u\beta \frac{F^{ij}}{F}\frac{\nn_iQ\nn_jQ}{Q}+A\cL u\\
\ge&(\frac{A}{2}u-C_{27}Q-C_{28})+A(\frac{1}{2}u-C_{29}Q)+(uA-2)\beta AQ\frac{F^{ii}}{F}(\nn_iu)^2>0
\end{align*}
for sufficiently small $Q$ and sufficiently large $A$. Thus by the maximum principle, we derive the lower bound of $Q$.
\end{proof}
\begin{lemma}\label{l4.9}
Let $X(\cdot,t)$ be the solution to the flow (\ref{1.10}) which encloses the origin for $t\in[0,T)$. If $F$ satisfies Assumption \ref{a1.3}, $F$ is an inverse concave function, and $\rho,u$ have positive lower and upper bound,
then there is a positive constant $C_{30}$ depending on the initial hypersurface,  $\Vert\widetilde G\Vert_{C^0}$, $\Vert\widetilde G\Vert_{C^1}$ and $\beta$,  such that
	$$F\leq C_{30}.$$
\end{lemma}
\textbf{Remark:} $\sigma_{k}^\frac{1}{k}$ is an inverse concave function. More examples can be seen in \cite{B4}.
\begin{proof}
Let $P=u\wg F^\beta$, $\hat G=u\wg$, $\cL=\p_t-\beta\frac{P}{F}F^{ij}\nn_{ij}^2-P(\frac{d_\nu \hat G(e_i)}{\hat G}-\frac{1}{u}\nn^i\Phi)\nn_i$, $\theta=\frac{P}{u-a}$, where $a=\frac{1}{2}\min u$. Similar to Lemma \ref{l4.7},  it suffice to get the upper bound of $P$ or $\theta$. At the spatial maximum point of $\theta$, we choose the normal coordinate system such that
$$h_j^i=h_{ij}=h_{ii}\delta_{ij},\qquad g_{ij}=\delta_{ij}.$$
Then $\nn_iu=h_i^l\nn_l\Phi=h_{ii}\nn_i\Phi$. Meanwhile, we have
$$0=\frac{u-a}{P}\nn\theta=\frac{\nn P}{P}-\frac{\nn u}{u-a}.$$
We can derive
\begin{equation}\label{4.14}
\begin{split}
\cL P=&\frac{P}{\hat G}\(d_X\hat G(X)(1-\frac{P}{u})
+d_\nu\hat G(u\nn(\frac{P}{u}))\)
+\beta\frac{P}{F}F^{ij}\(u\nn^2_{ij}(\frac{P}{u})-(1-\frac{P}{u})h^i_j\\
&+2h_{il}\nn^l(\frac{P}{u})\nn_j\Phi\)-\beta\frac{P}{F}F^{ij}\nn_{ij}^2P-P(\frac{d_\nu \hat G(e_i)}{\hat G}-\frac{1}{u}\nn^i\Phi)\nn_iP\\
=&P(1-\frac{P}{u})\(\frac{d_X\hat G(X)}{\hat G}-\beta\)-\frac{P^2d_\nu\hat G(\frac{\nn u}{u})}{\hat G}-\beta\frac{P^2}{F}F^{ij}\frac{\nn^2_{ij}u}{u}+\frac{P}{u}\nn^i\Phi\nn_iP\\
=&P(1-\frac{P}{u})\(\frac{d_X\hat G(X)}{\hat G}-\beta\)-\beta\frac{P^2}{u}+\beta\frac{P^2}{F}F^{ij}h^2_{ij}+\frac{P^2}{u}\nn^i\Phi\frac{d_X\hat G(e_i)}{\hat G}.
\end{split}
\end{equation}
Note that
\begin{align*}
\frac{d_\nu \hat G(e_i)}{\hat G}-\frac{1}{u}\nn^i\Phi=\frac{d_\nu \wg(e_i)}{\wg}
\end{align*}
and there is a difference between $\cL$ in this lemma and $\cL$ in the above lemma. By (\ref{4.13}), we have
\begin{equation}\label{4.15}
	\cL u
=u-(1+\beta)P+P\frac{d_X\wg(\nn\Phi)}{\wg}+\beta uP\frac{F^{ij}(h^2)_{ij}}{F}.
\end{equation}
Then by (\ref{4.14}) and (\ref{4.15}), we get
\begin{align*}
\frac{u-a}{P}\cL\theta=&\frac{\cL P}{P}-\frac{\cL u}{u-a}\\
\le&C_{31}+C_{32}P-\beta\frac{a}{u-a}P\frac{F^{ij}(h^2)_{ij}}{F}\\
\le&C_{31}+C_{32}P-\beta\frac{a}{u-a}PF\le0
\end{align*}
for sufficiently large $F$, where we have used $F^{ij}(h^2)_{ij}\ge F^2$ since $F$ is inverse concave. Thus by the maximum principle, we derive the upper bound of $F$.
\end{proof}
In summary,
by Lemma \ref{l4.3}, \ref{4.7}-\ref{4.9}, we can get the bounds of $u$ and $F$.
\begin{corollary}\label{c4.10}
Let $X(\cdot,t)$ be the solution to the flow (\ref{1.9}) which encloses the origin for $t\in[0,T)$. If $F$ satisfies Assumption \ref{a1.3}, $G^\frac{1}{\beta}$ satisfies (\rmnum{2})-(\rmnum{5}) of Assumption \ref{a1.4}, $(Q-1)|_{M_0}$ has a sign and $\rho$ has positive lower and upper bound,
then there are positive constants $C_{33}$, $C_{34}$ depending on the initial hypersurface, $\Vert G\Vert_{C^0}$, $\Vert G\Vert_{C^1}$ and $\beta$,  such that
$$\frac{1}{C_{33}}\leq u\leq C_{33}\quad\text{and}\quad\frac{1}{C_{34}}\leq F\leq C_{34}.$$
\end{corollary}
\begin{corollary}\label{c4.11}
Let $X(\cdot,t)$ be the solution to the flow (\ref{1.10}) which encloses the origin for $t\in[0,T)$. If $F=\sigma_{k}^\frac{1}{k}$, $G^\frac{1}{\beta}$ satisfies one of Assumption \ref{a1.4}, $(Q-1)|_{M_0}$ has a sign and $\rho$ has positive lower and upper bound,
then there are positive constants $C_{33}$, $C_{34}$ depending on the initial hypersurface, $\Vert\wg\Vert_{C^0}$, $\Vert\wg\Vert_{C^1}$ and $\beta$,  such that
	$$\frac{1}{C_{33}}\leq u\leq C_{33}\quad\text{and}\quad\frac{1}{C_{34}}\leq F\leq C_{34}.$$
\end{corollary}

\section{The $C^2$ estimates and proofs of main theorems}
We first derive the $C^2$ estimates of flows (\ref{1.9}) and (\ref{1.10}) respectively.
\begin{lemma}\label{l5.1}
Let $X(\cdot,t)$ be a smooth, closed and star-shaped solution to the flow (\ref{1.9}) which encloses the origin for $t\in[0,T)$. If $F$ satisfies Assumption \ref{a1.3}, $\vert D\rho\vert$ has an upper bound and $\rho$, $F$ has positive lower and upper bounds, $uG$ is log-convex, and $\frac{d_\nu G(\nu)}{G}\le\beta$,
then there is a  positive constant $C_{35}$ depending on the initial hypersurface, $\Vert G\Vert_{C^0}$, $\Vert G\Vert_{C^1}$, $\Vert G\Vert_{C^2}$ and $\beta$,  such that
 the principal curvatures of $X(\cdot,t)$ are uniformly bounded from above $$\k_i(\cdot,t)\leq C_{35} \text{ \qquad } \forall1\le i\le n,$$
and hence, are compactly contained in $\Gamma$, in view of Corollary \ref{c4.10}.
\end{lemma}
\begin{proof}
First, we shall prove that $\k_i$  is bounded from above by a positive constant. The principal curvatures of $M_t$ are the eigenvalues of $\{h_{il}g^{lj}\}$.

Define the functions
\begin{gather}\label{5.1}
	W(x,t)=\max\{h_{ij}(x,t)\xi^i\xi^j: g_{ij}(x)\xi^i\xi^j=1\},\\
	p(u)=-\log(u-\frac{1}{2}\min u),
\end{gather}
and
\begin{equation}\label{5.3}
	\theta=\log W+p(u)+N\rho,
\end{equation}
where $N$ will be chosen later. Note that
\begin{equation}\label{5.4}
	1+p'u=\frac{-\frac{1}{2}\min u}{u-\frac{1}{2}\min u}<0.
\end{equation}
We wish to bound $\theta$ from above. Thus, suppose $\theta$ attains a maximal value at $(\xi_0,t_0)\in M\times(0,T_0]$, $T_0<T^*$. Choose Riemannian normal coordinates $\{e_i\}$ in $(\xi_0,t_0)$, such that at this point we have
\begin{equation}\label{5.5}
	g_{ij}=\delta_{ij},\quad h^i_j=h_{ij}=\k_i\delta_{ij},\quad\k_1\geq\cdots\geq\k_n.
\end{equation}
Since $W$ is only continuous in general, we need to find a differentiable version instead. Set
$$\widetilde{W}=\frac{h_{ij}\tilde{\xi}^i\tilde{\xi}^j}{g_{ij}\tilde{\xi}^i\tilde{\xi}^j},$$
where $\tilde{\xi}=(\tilde{\xi}^i)=(1,0,\cdots,0).$

At $(\xi_0,t_0)$ we have
\begin{equation}\label{5.6}
	h_{11}=h^1_1=\k_1=W=\widetilde{W}
\end{equation}
and in a neighborhood of $(\xi_0,t_0)$ there holds
$$\widetilde{W}\leq W.$$
Using $h^1_1=h_{1k}g^{k1}$, we find that at $(\xi_0,t_0)$
$$\frac{d\widetilde{W}}{dt}=\frac{dh_1^1}{dt}$$
and the spatial derivatives also coincide. Replacing $\theta$ by $\widetilde{\theta}=\log \widetilde{W}+p(u)+N\widetilde\rho$, we see that $\widetilde{\theta}$ attains a maximal value at $(\xi_0,t_0)$, where $\widetilde{W}$ satisfies the same differential equation in this point as $h_1^1$. Thus, without loss of generality, we may pretend $h_1^1$ to be a scalar and $\theta$ to be given by
\begin{equation}\label{5.7}
	\theta=\log h_1^1+p(u)+N\rho.
\end{equation}
In order to calculate the evolution equations of $\theta$, we should deduce the evolution equations of $h_1^1,$ $u$ and $\rho$. Let $$J=F^{-\beta}.$$
By (\ref{3.1}), we have
\begin{equation}\label{5.8}
	\p_th^1_{1}=-u\nn_1\nn^1(GJ)-(GJ-1)h^1_1-2\rho h_{11}\nn^1(GJ)\nn^1\rho.
\end{equation}
Remark that $$J^{ij}=\frac{\p J}{\p h_{ij}},\qquad J^{ij,mn}=\frac{\p^2J}{\p h_{ij}\p h_{mn}}.$$
In order to make the following calculations easier, we define $\bar G=uG$. Then by (\ref{2.13}),
\begin{align*}
\bar G_{;11}=Gu_{;11}+2u_{;1}G_{;1}+uG_{;11}=&uG_{;11}+2u_{;1}G_{;1}+G(\nn^k\Phi\nn_kh_{11}+h_{11}-uh_{11}^2),\\
\bar G_{;1}J_{;1}=&Gu_{;1}J_{;1}+uG_{;1}J_{;1}.
\end{align*}
We can replace $G$ by $\bar G$ in (\ref{5.8}),
\begin{equation}\label{5.9}
\begin{split}
\p_th^1_{1}=&-uJG_{;11}-2uG_{;1}J_{;1}-uGJ_{;11}-(GJ-1)h^1_1-2\rho h_{11}\nn^1(GJ)\nn^1\rho\\
=&-J\bar G_{;11}-\bar GJ_{;11}-2\bar G_{;1}J_{;1}+\frac{\bar GJ}{u}\nn^k\Phi\nn_kh_{11}+h_{11}-\bar GJh_{11}^2,
\end{split}
\end{equation}
where we have used $\nn_1u=h_1^1\nn_1\Phi=\rho h_1^1\nn_1\rho$. Note that
\begin{equation}\label{5.10}
	\begin{split}
\bar G_{;k}=&d_X\bar G(e_k)+d_\nu\bar  G(\bar\nn_k\nu)=d_X\bar G(e_k)+d_\nu \bar G(h_k^ie_i),\\
\bar G_{;11}=&d_Xd_X\bar G(e_1,e_1)+2h_{11}d_\nu d_X\bar G(e_1,e_1)-h_{11}d_X\bar G(\nu)\\
&+h_{11}^2d_\nu d_\nu\bar  G(e_1,e_1)+\nn_ih_{11}d_\nu\bar  G(e_i)-h_{11}^2d_\nu\bar  G(\nu)\\
\ge&-C_{36}-C_{37}h_{11}+h_{11}^2d_\nu d_\nu\bar  G(e_1,e_1)+\nn_ih_{11}d_\nu\bar  G(e_i)-h_{11}^2d_\nu\bar  G(\nu).
	\end{split}
\end{equation}
By (\ref{2.11}) and Ricci identity, we have
\begin{equation*}
	\begin{split}
\nn_1\nn_1h_{ij}= h_{11;ij}+{R_{j11}}^ah_{ai}+{R_{j1i}}^ah_{a1}=h_{11;ij}+h_{ai}h_{aj}h_{11}-h_{ij}h_{11}^2,
	\end{split}
\end{equation*}
where we assume the indexes $i$ and $j$ are symmetric. Since
$$F_{;11}=F^{ij,kl}\nn_1h_{ij}\nn_1h_{kl}+F^{ij}\nn^2_{11}h_{ij},$$
we get
\begin{equation*}
	\begin{split}
\frac{\p h_1^1}{\p t}\le&C_{38}+C_{39}h_{11}-Jh_{11}^2d_\nu d_\nu\bar  G(e_1,e_1)-J\nn_ih_{11}d_\nu \bar G(e_i)+Jh_{11}^2d_\nu \bar G(\nu)\\
&+2\beta\(d_X\bar G(e_1)+d_\nu \bar G(h_{11}e_1)\)\frac{J}{F}F_{;1}-\bar G(-\beta\frac{J}{F}F_{;11}+\beta(\beta+1)\frac{J}{F^2}(\nn_1F)^2)\\
&+\frac{\bar GJ}{u}\nn^k\Phi\nn_kh_{11}-\bar GJh_{11}^2\\
\le&C_{38}+C_{39}h_{11}-Jh_{11}^2d_\nu d_\nu\bar  G(e_1,e_1)-J\nn_ih_{11}d_\nu \bar G(e_i)+Jh_{11}^2d_\nu \bar G(\nu)\\
&+2\beta\(d_X\bar G(e_1)+d_\nu \bar G(h_{11}e_1)\)\frac{J}{F}F_{;1}+\beta\frac{\bar GJ}{F}F^{ij,kl}\nn_1h_{ij}\nn_1h_{kl}\\
&+\beta\frac{\bar GJ}{F}F^{ij}(h_{11;ij}+h_{ai}h_{aj}h_{11}-h_{ij}h_{11}^2)-\beta(\beta+1)\frac{\bar GJ}{F^2}(\nn_1F)^2\\
&+\frac{\bar GJ}{u}\nn^k\Phi\nn_kh_{11}-\bar GJh_{11}^2\\
\le&C_{38}+C_{39}h_{11}-Jh_{11}^2d_\nu d_\nu\bar  G(e_1,e_1)-\bar GJh_{11}^2\(\beta+1-\frac{d_\nu \bar G(\nu)}{\bar G}\)\\
&-J\nn_ih_{11}(d_\nu \bar G(e_i)-\frac{\bar G}{u}\nn^i\Phi)+2\beta\(d_X\bar G(e_1)+d_\nu \bar G(h_{11}e_1)\)\frac{J}{F}F_{;1}\\
&+\beta\frac{\bar GJ}{F}F^{ij,kl}\nn_1h_{ij}\nn_1h_{kl}+\beta\frac{\bar GJ}{F}F^{ij}h_{11;ij}+\beta\frac{\bar GJ}{F}F^{ij}(h_{ij})^2h_{11}\\
&-\beta(\beta+1)\frac{\bar GJ}{F^2}(\nn_1F)^2.
 	\end{split}
\end{equation*}
Define the operator $\mathcal{L}$ by
\begin{equation}\label{5.11}
	\mathcal{L}=\p_t-\beta\frac{\bar GJ}{F}F^{ij}\nn_{ij}^2+J(d_\nu \bar G(e_i)-\frac{\bar G}{u}\nn^i\Phi)\nn_i.
\end{equation}
We claim
$$2\beta\(d_X\bar G(e_1)+d_\nu \bar G(h_{11}e_1)\)\frac{J}{F}F_{;1}-Jh_{11}^2d_\nu d_\nu\bar  G(e_1,e_1)-\beta(\beta+1)\frac{\bar GJ}{F^2}(\nn_1F)^2\le C.$$
In fact, since $uG$ is log-convex, $d_\nu d_\nu(uG)-\frac{(d_\nu(uG))^2}{uG}$ is semi-positive definite. It is easy to be proved by the following inequalities.
\begin{align*}
2\beta d_\nu \bar G(h_{11}e_1)\frac{J}{F}F_{;1}&\le\frac{(d_\nu(\bar G)(h_{11}e_1))^2}{\bar G}J+\beta^2\frac{\bar GJ}{F^2}(\nn_1F)^2,\\
&\le Jh_{11}^2d_\nu d_\nu\bar  G(e_1,e_1)+\beta^2\frac{\bar GJ}{F^2}(\nn_1F)^2,\\
2\beta d_X\bar G(e_1)\frac{J}{F}F_{;1}&\le\beta\frac{\bar GJ}{F^2}(\nn_1F)^2
+\frac{1}{\beta}\frac{(d_X\bar G(e_1))^2J}{\bar G}.
\end{align*}
Thus, by the claim we deduce
\begin{equation}\label{5.12}
	\begin{split}
\cL h_1^1\le&C_{38}+C_{39}h_{11}-\bar GJh_{11}^2\(\beta+1-\frac{d_\nu \bar G(\nu)}{\bar G}\)+\beta\frac{\bar GJ}{F}F^{ij,kl}\nn_1h_{ij}\nn_1h_{kl}\\
&+\beta\frac{\bar GJ}{F}F^{ij}(h_{ij})^2h_{11}.
	\end{split}
\end{equation}
In addition, by (\ref{3.1}), we deduce
\begin{equation}\label{5.13}
	\begin{split}
\cL u=&u(GJ-1)-u\la X,\nn(GJ)\ra-\beta\frac{\bar GJ}{F}F^{ij}\nn_{ij}^2u+J(d_\nu \bar G(e_i)-\frac{\bar G}{u}\nn^i\Phi)\nn_iu\\
=&C_{40}-uJ\nn^i\Phi\nn_iG-\bar G\nn^i\Phi\nn_iJ-\beta\frac{\bar GJ}{F}F^{ij}(g^{kl}\nabla_kh_{ij}\nabla_l\Phi+h_{ij}-(h^2)_{ij}u)\\
&+J(d_\nu \bar G(e_i)-\frac{\bar G}{u}\nn^i\Phi)\nn_iu\\
=&C_{40}-J\nn^i\Phi\nn_i\bar G+\beta\frac{u\bar GJ}{F}F^{ij}(h^2)_{ij}+Jd_\nu \bar G(e_i)\nn_iu\\
=&C_{41}+\beta\frac{u\bar GJ}{F}F^{ij}(h^2)_{ij},
	\end{split}
\end{equation}
where we have used $X=\bar\nn\Phi$ and $\nn_i\Phi h_i^ie_i=u_ie_i=\nn u$.
By (\ref{2.16}), (\ref{l2.3}) and (\ref{2.17}), we have
\begin{equation}\label{5.14}
	\begin{split}
\cL\rho=&(GJ-1)\rho-\beta\frac{\bar GJ}{F}F^{ij}(-\omega^{-1}h_{ij}+\frac{1}{\rho}g_{ij}-\frac{1}{\rho}\rho_{i}\rho_{j})+J(d_\nu \bar G(e_i)-\frac{\bar G}{u}\nn^i\Phi)\nn_i\rho\\
\le&C_{42}-\beta\frac{\bar GJ}{F\rho}F^{ij}g_{ij}+\beta\frac{\bar GJ}{F\rho}F^{ij}g_{ij}|\nn\rho|^2\\
=&C_{42}-\beta\frac{\bar GJ}{F\rho}F^{ij}g_{ij}\frac{1}{\om^2}.
	\end{split}
\end{equation}
If $\k_1$ is sufficiently large, the combination of (\ref{5.12}), (\ref{5.13}) and (\ref{5.14}) gives
\begin{equation}\label{5.15}
	\begin{split}
\cL\theta=&\frac{\cL h_1^1}{h_1^1}+\beta\frac{\bar GJ}{F}F^{ij}\nn_i(\log h_1^1)\nn_j(\log h_1^1)+p'\cL u-p''\beta\frac{\bar GJ}{F}F^{ij}\nn_iu\nn_ju\\
&+N\cL\rho\\
\le&C_{43}-\bar GJh_{11}\(\beta+1-\frac{d_\nu \bar G(\nu)}{\bar G}\)+\beta\frac{\bar GJ}{Fh_{11}}F^{ij,kl}\nn_1h_{ij}\nn_1h_{kl}\\
&+\beta\frac{\bar GJ}{F}F^{ij}(\nn_i(\log h_1^1)\nn_j(\log h_1^1)-p''\nn_iu\nn_ju)+(p'u+1)\beta\frac{\bar GJ}{F}F^{ij}(h^2)_{ij}\\
&-N\beta\frac{\bar GJ}{F\rho}F^{ij}g_{ij}\frac{1}{\om^2}.
\end{split}
\end{equation}
Due to the concavity of $F$ it holds that
\begin{equation}\label{5.16}
F^{kl,rs}\xi_{kl}\xi_{rs}\leq\sum_{k\neq l}\frac{F^{kk}-F^{ll}}{\k_k-\k_l}\xi_{kl}^2\leq\frac{2}{\k_1-\k_n}\sum_{k=1}^{n}(F^{11}-F^{kk})\xi_{1k}^2
\end{equation}
for all symmetric matrices $(\xi_{kl})$; cf. \cite{GC2}. Furthermore, we have
\begin{equation}\label{5.17}
F^{11}\leq\cdots\leq F^{nn};
\end{equation}
cf. \cite{EH}. In order to estimate (\ref{5.15}), we distinguish between two cases.

Case 1: $\k_n<-\eps_1\k_1$, $0<\eps_1<\frac{1}{2}$. Then
\begin{equation}\label{5.18}
F^{ij}(h^2)_{ij}\geq F^{nn}\k_n^2\geq\frac{1}{n}F^{ij}g_{ij}\k_n^2\geq\frac{1}{n}F^{ij}g_{ij}\eps_1^2\k_1^2.
\end{equation}
We use $\nn\theta=0$ to obtain
\begin{equation}\label{5.19}
F^{ij}\nn_i(\log h_1^1)\nn_j(\log h_1^1)=p'^2F^{ij}u_{;i}u_{;j}+2Np'F^{ij}u_{;i}\rho_{;j}+N^2F^{ij}\rho_{;i}\rho_{;j}.
\end{equation}
In this case, the concavity of $f$ implies that
\begin{equation}\label{5.20}
\beta\frac{\bar GJ}{Fh_{11}}F^{ij,mn}h_{ij;1}h_{mn;1}\leq0.
\end{equation}
By (\ref{2.13}) and note that $p'<0$, we have
\begin{equation}\label{5.21}
	\begin{split}
2N\beta p'\frac{\bar GJ}{F}F^{ij}u_{;i}\rho_{;j}=&2N\beta p'\rho\frac{\bar GJ}{F}F^{ij}\k_i\rho_{;i}\rho_{;j}.	
\end{split}
\end{equation}
By \cite{UJ} Lemma 3.3, we know $\frac{H}{n}\geq \frac{F}{F(1,\cdots,1)}\geq C_{44}$, where $H=\sum_{i}^{n}\k_i$ is the mean curvature. Thus $(n-1)\k_1+\k_n\ge H\ge C_{44}$. We can derive $\k_i\ge\k_n\ge C_{44}-(n-1)\k_1$. For fixed $i$, if $\k_i\geq0$, we derive $$2N\beta p'\rho\frac{\bar GJ}{F}F^{ij}\k_i\rho_{;i}\rho_{;j}\leq0.$$ If $\k_i<0$, we have
$$2N\beta p'\rho\frac{\bar GJ}{F}F^{ij}\k_i\rho_{;i}\rho_{;j}\leq2N\beta p'\rho\frac{\bar GJ}{F}\k_iF^{ij}g_{ij} \leq C_{45}(C_{44}-(n-1)\k_1)F^{ij}g_{ij},$$
 where we have used $F^{ij}\rho_{;i}\rho_{;j}\leq F^{ij}g_{ij}\vert\nn\rho\vert^2\leq F^{ij}g_{ij}$.

Without loss of generality, we can assume that $\k_k>0$ and $\k_{k+1}\le0$, then
\begin{equation}\label{5.22}
2N\beta p'\rho\frac{\bar GJ}{F}F^{ij}\k_i\rho_{;i}\rho_{;j}\leq(n-k)C_{45}(C_{44}-(n-1)\k_1)F^{ij}g_{ij}.
\end{equation}

Since $p'^2=p''$ and $1+p'u<0$, by the combination of (\ref{l2.3}), (\ref{5.18}), (\ref{5.19}), (\ref{5.20}) and (\ref{5.22}), in this case (\ref{5.15}) becomes
\begin{equation}\label{5.23}
	\begin{split}
\cL\theta\leq\beta\frac{\bar GJ}{F} F^{ij}g_{ij}\(\frac{1}{n}\eps_1^2\k_1^2(1+p'u)+C_{46}\k_1+C_{47}\)-\bar GJh_{11}\(\beta+1-\frac{d_\nu \bar G(\nu)}{\bar G}\)+C_{48},
	\end{split}
	\end{equation}
which is negative for large $\k_1$. We also use $\beta+1\ge\frac{d_\nu \bar G(\nu)}{\bar G}$ and $F^{ij}g_{ij}\ge F(1,\cdots,1)$ here since $F$ is concave.

Case 2: $\k_n\geq-\eps_1\k_1$. Then
$$\frac{2}{\k_1-\k_n}\sum_{k=1}^{n}(F^{11}-F^{kk})(h_{11;k})^2k_1^{-1}\leq\frac{2}{1+\eps_1}\sum_{k=1}^{n}(F^{11}-F^{kk})(h_{11;k})^2k_1^{-2}.$$
We deduce further
\begin{equation}\label{5.24}
	\begin{split}
F^{ij}&\nn_i(\log h_1^1)\nn_j(\log h_1^1)+\frac{2}{\k_1-\k_n}\sum_{k=1}^{n}(F^{11}-F^{kk})(h_{11;k})^2k_1^{-1}\\
\leq&\frac{2}{1+\eps_1}\sum_{k=1}^{n}F^{11}(\log h_1^1)_{;k}^2-\frac{1-\eps_1}{1+\eps_1}\sum_{k=1}^{n}F^{kk}(\log h_1^1)_{;k}^2\\
\leq&\sum_{k=1}^{n}F^{11}(\log h_1^1)_{;k}^2\\
\leq&F^{11}(p'^2\vert\nn u\vert^2+2Np'<\nn u,\nn\rho>+N^2\vert\nn\rho\vert^2),
	\end{split}
\end{equation}
where we have used $F^{kk}\geq F^{11}$ in the second inequality.
Note that
\begin{gather}
\beta(1+p'u)\frac{\bar GJ}{F}F^{ij}(h^2)_{ij}\leq\beta(1+p'u)\frac{\bar GJ}{F}F^{11}\k_1^2,\label{5.25}\\
2Np'F^{11}<\nn u,\nn\rho>=2Np'\rho F^{11}\k_i\rho_{;i}^2\leq-2\eps_1Np'\rho F^{11}\rho_{;i}^2\k_1,\label{5.26}\\
-\beta\frac{\bar GJ}{F}p''F^{ij}\nn_iu\nn_ju+\beta\frac{\bar GJ}{F}p''F^{11}\vert\nn u\vert^2\leq0.\label{5.27}
\end{gather}
By the combination of (\ref{5.24})-(\ref{5.27}),  (\ref{5.15}) becomes
\begin{equation}\label{3.38}
\begin{split}
\cL\theta\leq& C_{49}-\bar GJh_{11}\(\beta+1-\frac{d_\nu \bar G(\nu)}{\bar G}\)+\beta\frac{\bar GJ}{F}F^{11}\((1+p'u)k_1^2+C_{50}\k_1+C_{51}\)\\
&-\beta \frac{\bar GJ}{F}F^{ij}g_{ij}N\frac{1}{\rho\omega^2},
\end{split}
\end{equation}
which is negative for large $\k_1$ after fixing $N_0$ large enough to ensure that
$$\beta F(1,\cdots,1) \frac{\bar GJ}{F}N_0\frac{1}{\rho\omega^2}\geq C_{49}$$
by $\beta+1\ge\frac{d_\nu \bar G(\nu)}{\bar G}$ and $F^{ij}g_{ij}\ge F(1,\cdots,1)$.
Hence in this case any $N\geq N_0$ yields an upper bound for $\k_1$.

In conclusion, $\k_i\leq C_{35}$, where $C_{35}$ depends on the initial hypersurface, $\Vert G\Vert_{C^0}$, $\Vert G\Vert_{C^1}$, $\Vert G\Vert_{C^2}$ and $\beta$. Since $F$ is uniformly continuous on the convex cone $\overline{\Gamma}$, and $F$ is bounded from below by a positive constant. Corollary \ref{c4.10} and Assumption \ref{a1.3} imply that $\k_i$ remains in a fixed compact subset of $\Gamma$, which is independent of $t$.
\end{proof}
Before analyzing the case of $\Psi(s)=-s^{-1}$ and $F=\sigma_{k}^\frac{1}{k}$, we first state a crucial lemma (see \cite{GLL} Lemma 3.2).
\begin{lemma}\cite{GLL}\label{l5.2}
If $\{h_j^i\}\in\Gamma_k^+$ and $k\ge2$, we have the following inequality.
\begin{align*}
\sigma_{k}^{pq,rs}\nn_ih_{pq}\nn_ih_{rs}\le-\sigma_{k}\(\frac{\nn_i\sigma_{k}}{\sigma_{k}}-\frac{\nn_iH}{H}\)\((\frac{2-k}{k-1})\frac{\nn_i\sigma_{k}}{\sigma_{k}}-(\frac{k}{k-1})\frac{\nn_iH}{H}\),
\end{align*}
where $H=\sum k_i$ is the mean curvature.
\end{lemma}
\begin{lemma}\label{l5.3}
	Let $X(\cdot,t)$ be a smooth, closed and star-shaped solution to the flow (\ref{1.10}) which encloses the origin for $t\in[0,T)$. If $F=\sigma_{k}^\frac{1}{k}$, $\vert D\rho\vert$ has an upper bound and $\rho$, $F$ has positive lower and upper bounds, $\frac{d_\nu \wg(\nu)}{\wg}+\beta\ge0$, and either (1) $u\widetilde{G}$ is log-concave with respect to $\nu$ for $\beta\ge k$, or (2) $d_\nu d_\nu(u\wg)-\frac{\beta}{\beta-1}\frac{(d_\nu(u\wg))^2}{u\wg}$ is semi-negative definite for $0<\beta<1$,
	then there is a  positive constant $C_{52}$ depending on the initial hypersurface, $\Vert \wg\Vert_{C^0}$, $\Vert \wg\Vert_{C^1}$, $\Vert \wg\Vert_{C^2}$ and $\beta$,  such that
	the principal curvatures of $X(\cdot,t)$ are uniformly bounded from above $$\k_i(\cdot,t)\leq C_{52} \text{ \qquad } \forall1\le i\le n,$$
	and hence, are compactly contained in $\Gamma$, in view of Corollary \ref{c4.11}.
\end{lemma}
\begin{proof}
As $\{h_i^j\}\in\Gamma_k^+$, $k\ge2$, we have $H>0$, $\sigma_{2}>0$, then
$$H^2=2\sigma_{2}+|A|^2>|A|^2.$$
So we only need to derive the upper bound of $H$ to get the upper bound of $\k_i$.

We modify the auxiliary function posed in \cite{LXZ},
$$\theta=\frac{H}{\hat{G}J-a},$$
where $\hat G=u\widetilde{G}$, $J=F^{\beta}$ and $a=\frac{1}{2}\inf_{M\times[0,T)}\hat GJ$. Let $P=\hat{G}J$. By Lemma \ref{l3.3} we have
\begin{equation}\label{5.29}
	\p_th^i_{j}=u\nn_j\nn^i(\tilde GJ)+(\tilde GJ-1)h^i_j+2 h^i_{l}\nn^l(GJ)\nn_j\Phi.
\end{equation}
Similar to the proof of Lemma \ref{l5.1},  Choose Riemannian normal coordinates $\{e_i\}$ in $(\xi_0,t_0)$, such that in this point we have
\begin{equation*}
	g_{ij}=\delta_{ij},\quad h^i_j=h_{ij}=\k_i\delta_{ij},\quad\k_1\geq\cdots\geq\k_n.
\end{equation*}
We have
\begin{align*}
\Delta\hat G=&u\Delta\widetilde G+2u_{;i}\widetilde G_{;i}+\widetilde  G(\nn^k\Phi\nn_kH+H-u|A|^2),\\
	\hat G_{;i}J_{;i}=&\widetilde Gu_{;i}J_{;i}+u\widetilde G_{;i}J_{;i}.
\end{align*}
Thus we can derive
\begin{equation}\label{5.30}
	\begin{split}
		\p_tH=&uJ\Delta\widetilde G+2u\widetilde G_{;i}J_{;i}+u\widetilde G\Delta J+(\widetilde GJ-1)H+2h_{ii}\nn^i(\widetilde GJ)\nn^i\Phi\\
		=&J\Delta\hat G-\frac{\hat GJ}{u}\nn^k\Phi\nn_kH+\hat GJ|A|^2+2\hat G_{;i}J_{;i}+\hat G\Delta J-H,
	\end{split}
\end{equation}
where we have used $\nn_iu=h_i^i\nn_i\Phi$, $|A|^2=h_j^ih^j_i$. Note that
\begin{equation}\label{5.31}
	\begin{split}
		\hat G_{;k}=&d_X\hat G(e_k)+d_\nu\hat G(\bar\nn_k\nu)=d_X\hat  G(e_k)+d_\nu \hat G(h_k^ie_i),\\
\Delta\hat G=&d_Xd_X\hat G(e_k,e_k)+2h_{kk}d_\nu d_X\hat  G(e_k,e_k)-Hd_X\hat G(\nu)\\
		&+h_{kk}^2d_\nu d_\nu\hat G(e_k,e_k)+\nn_kHd_\nu\hat  G(e_k)-|A|^2d_\nu\hat  G(\nu)\\
		\le&C_{53}+C_{54}H+h_{kk}^2d_\nu d_\nu\hat  G(e_k,e_k)+\nn_kHd_\nu\hat G(e_k)-|A|^2d_\nu\hat G(\nu).
	\end{split}
\end{equation}
By (\ref{2.11}) and Ricci identity, we have
\begin{equation*}
	\begin{split}
\Delta h_{ij}= H_{;ij}+{R_{jkk}}^ah_{ai}+{R_{jki}}^ah_{ak}=H_{;ij}+h_{ai}h_{aj}H-h_{ij}|A|^2,
	\end{split}
\end{equation*}
where we assume the indexes $i$ and $j$ are symmetric. Since
$$\Delta J=J^{pq,rs}\nn_ih_{pq}\nn_ih_{rs}+J^{pq}\Delta h_{pq}=J^{pq,rs}\nn_ih_{pq}\nn_ih_{rs}+J^{pq}(H_{;ij}+h_{ai}h_{aj}H-h_{ij}|A|^2),$$
we get
\begin{equation*}
	\begin{split}
\frac{\p H}{\p t}\le&C_{53}+C_{54}H+Jh_{kk}^2d_\nu d_\nu\hat G(e_k,e_k)+J\nn_kHd_\nu \hat G(e_k)-J|A|^2d_\nu \hat G(\nu)\\
&+2\(d_X\hat G(e_k)+d_\nu \hat G(h_{kk}e_k)\)J_{;k}+\hat G(J^{pq,rs}\nn_ih_{pq}\nn_ih_{rs}+J^{pq}\Delta h_{pq})\\
&-\frac{\hat GJ}{u}\nn^k\Phi\nn_kH+\hat GJ|A|^2.
	\end{split}
\end{equation*}
Define the operator $\mathcal{L}$ by
\begin{equation}\label{5.32}
	\mathcal{L}=\p_t-\frac{\beta}{k}\frac{\hat GJ}{\sigma_{k}}\sigma_{k}^{ij}\nn_{ij}^2-J(d_\nu \hat G(e_i)-\frac{\hat G}{u}\nn^i\Phi)\nn_i.
\end{equation}
Thus, we can deduce
\begin{equation}\label{5.33}
	\begin{split}
		\cL H\le&C_{53}+C_{54}H+Jh_{kk}^2d_\nu d_\nu\hat G(e_k,e_k)-\hat GJ|A|^2\(\beta-1+\frac{d_\nu \hat G(\nu)}{\hat G}\)\\
		&+2\(d_X\hat G(e_k)+d_\nu \hat G(h_{kk}e_k)\)J_{;k}+\hat GJ^{pq,rs}\nn_ih_{pq}\nn_ih_{rs}
		+\frac{\beta}{k}\frac{\hat GJ}{\sigma_{k}}\sigma_{k}^{ij}(h_{ij})^2H.
	\end{split}
\end{equation}
Using Lemma \ref{l5.2}, we obtain
\begin{equation}\label{5.34}
\begin{split}
J^{pq,rs}\nn_ih_{pq}\nn_ih_{rs}\le&-\frac{\beta}{k}J\(\frac{k}{\beta}\frac{\nn_iJ}{J}-\frac{\nn_iH}{H}\)\(\frac{2-k}{k-1}\frac{k}{\beta}\frac{\nn_iJ}{J}-\frac{k}{k-1}\frac{\nn_iH}{H}\)\\
&-\(\frac{k}{\beta}-1\)\frac{|\nn J|^2}{J}\\
=&\(1-\frac{k}{\beta(k-1)}\)\frac{|\nn J|^2}{J}+\frac{2}{k-1}\nn_iJ\frac{\nn_iH}{H}-\frac{\beta J}{k-1}\frac{|\nn H|^2}{H^2}.
\end{split}
\end{equation}
At the maximum point of $\theta$ on $M_t$, we have
\begin{equation}\label{5.35}
\frac{\nn_iH}{H}=\frac{\nn_iP}{P-a}=\frac{P}{P-a}\(\frac{\nn_iJ}{J}+\frac{d_X\hat G(e_i)}{\hat G}+\frac{d_\nu\hat G(h_{ii}e_i)}{\hat G}\).
\end{equation}
We claim that
\begin{align*}
\Rmnum{1}=&\frac{\hat G}{P}J^{pq,rs}\nn_ih_{pq}\nn_ih_{rs}+\frac{ J}{P}h_{kk}^2d_\nu d_\nu\hat G(e_k,e_k)+\frac{2}{P}\(d_X\hat G(e_k)+d_\nu \hat G(h_{kk}e_k)\)J_{;k}\\
\le& C+CH.
\end{align*}
In fact,

(1) If $u\widetilde{G}$ is log-concave, we have
$$d_\nu d_\nu\hat G(e_k,e_k)\le \frac{(d_\nu\hat G(e_k))^2}{\hat G}.$$
Inserting (\ref{5.35}) into (\ref{5.34}),
\begin{equation}
\begin{split}
\Rmnum{1}
\le&(1-\frac{k}{\beta(k-1)})\frac{|\nn J|^2}{J^2}+\frac{2}{k-1}\frac{\nn_iJ}{J}\frac{\nn_iH}{H}-\frac{\beta }{k-1}\frac{|\nn H|^2}{H^2}+\frac{(d_\nu\hat G(h_{ii}e_i))^2}{\hat G^2}\\
&+C\frac{|\nn J|}{J}+2\frac{d_\nu\hat G(h_{ii}e_i)}{\hat G}\frac{\nn_iJ}{J}\\
=&\(1-\frac{k}{\beta(k-1)}+\frac{2}{k-1}\frac{P}{P-a}-\frac{\beta}{k-1}(\frac{P}{P-a})^2\)\frac{|\nn J|^2}{J^2}\\
&+2\frac{d_\nu\hat G(h_{ii}e_i)}{\hat G}\frac{\nn_iJ}{J}\(1+\frac{1}{k-1}\frac{P}{P-a}-\frac{\beta}{k-1}(\frac{P}{P-a})^2\)\\
&+\(1-\frac{\beta}{k-1}(\frac{P}{P-a})^2\)\(\frac{d_\nu\hat G(h_{ii}e_i)}{\hat G}\)^2+C\frac{|\nn J|}{J}+CH.
\end{split}
\end{equation}
Note that
\begin{align*}
1-\frac{\beta}{k-1}(\frac{P}{P-a})^2<0
\end{align*}
by $\beta\ge k$ and $\frac{P}{P-a}>1$. We rewrite
\begin{align*}
\Rmnum{1}=&\(1-\frac{\beta}{k-1}(\frac{P}{P-a})^2\)\(\frac{d_\nu\hat G(h_{ii}e_i)}{\hat G}+\dfrac{1+\frac{1}{k-1}\frac{P}{P-a}-\frac{\beta}{k-1}(\frac{P}{P-a})^2}{1-\frac{\beta}{k-1}(\frac{P}{P-a})^2}\frac{\nn_iJ}{J}\)^2\\
&-\dfrac{\(1+\frac{1}{k-1}\frac{P}{P-a}-\frac{\beta}{k-1}(\frac{P}{P-a})^2\)^2}{1-\frac{\beta}{k-1}(\frac{P}{P-a})^2}\frac{|\nn J|^2}{J^2}\\
&+\(1-\frac{k}{\beta(k-1)}+\frac{2}{k-1}\frac{P}{P-a}-\frac{\beta}{k-1}(\frac{P}{P-a})^2\)\frac{|\nn J|^2}{J^2}+C\frac{|\nn J|}{J}\\
\le&-\frac{\frac{1}{k-1}\(\frac{k}{\beta}-(\frac{P}{P-a})^2\)}{1-\frac{\beta}{k-1}(\frac{P}{P-a})^2}\frac{|\nn J|^2}{J^2}+C\frac{|\nn J|}{J}+CH\\
\le& C_{55}+CH,
\end{align*}
where we have used Cauchy-Schwartz inequality, $\frac{k}{\beta}-(\frac{P}{P-a})^2<0$, $1-\frac{\beta}{k-1}(\frac{P}{P-a})^2<0$  in the last inequality since $\beta\ge k$ and $\frac{P}{P-a}>1$.

(2) If $d_\nu d_\nu(u\wg)-\frac{\beta}{\beta-1}\frac{(d_\nu(u\wg))^2}{u\wg}$ is semi-negative definite, we have
$$d_\nu d_\nu\hat{G}\le\frac{\beta}{\beta-1}\frac{(d_\nu\hat{G})^2}{\hat{G}}.$$
Inserting (\ref{5.35}) into (\ref{5.34}),
\begin{equation*}
	\begin{split}
		\Rmnum{1}
		\le&\(1-\frac{k}{\beta(k-1)}+\frac{2}{k-1}\frac{P}{P-a}-\frac{\beta}{k-1}(\frac{P}{P-a})^2\)\frac{|\nn J|^2}{J^2}\\
		&+2\frac{d_\nu\hat G(h_{ii}e_i)}{\hat G}\frac{\nn_iJ}{J}\(1+\frac{1}{k-1}\frac{P}{P-a}-\frac{\beta}{k-1}(\frac{P}{P-a})^2\)\\
		&+\(\frac{\beta}{\beta-1}-\frac{\beta}{k-1}(\frac{P}{P-a})^2\)\(\frac{d_\nu\hat G(h_{ii}e_i)}{\hat G}\)^2+C\frac{|\nn J|}{J}+CH.
	\end{split}
\end{equation*}
Note that
\begin{align*}
\frac{\beta}{\beta-1}-\frac{\beta}{k-1}(\frac{P}{P-a})^2<0
\end{align*}
by $\beta<1$ and $\frac{P}{P-a}>1$. We rewrite
\begin{align*}
	\Rmnum{1}=&\(\frac{\beta}{\beta-1}-\frac{\beta}{k-1}(\frac{P}{P-a})^2\)\(\frac{d_\nu\hat G(h_{ii}e_i)}{\hat G}+\dfrac{1+\frac{1}{k-1}\frac{P}{P-a}-\frac{\beta}{k-1}(\frac{P}{P-a})^2}{\frac{\beta}{\beta-1}-\frac{\beta}{k-1}(\frac{P}{P-a})^2}\frac{\nn_iJ}{J}\)^2\\
	&-\dfrac{\(1+\frac{1}{k-1}\frac{P}{P-a}-\frac{\beta}{k-1}(\frac{P}{P-a})^2\)^2}{\frac{\beta}{\beta-1}-\frac{\beta}{k-1}(\frac{P}{P-a})^2}\frac{|\nn J|^2}{J^2}\\
	&+\(1-\frac{k}{\beta(k-1)}+\frac{2}{k-1}\frac{P}{P-a}-\frac{\beta}{k-1}(\frac{P}{P-a})^2\)\frac{|\nn J|^2}{J^2}+C\frac{|\nn J|}{J}\\
	\le&-\frac{\frac{1}{(k-1)(\beta-1)}(\frac{a}{P-a})^2}{\frac{\beta}{\beta-1}-\frac{\beta}{k-1}(\frac{P}{P-a})^2}\frac{|\nn J|^2}{J^2}+C\frac{|\nn J|}{J}+CH\\
	\le& C_{55}+CH,
\end{align*}
where we have used Cauchy-Schwartz inequality and $0<\beta<1$  in the last inequality. In summary, we have proved this claim.

At the same time, we have
$$\beta-1+\frac{d_\nu \hat G(\nu)}{\hat G}\ge0.$$
Thus, by (\ref{5.33}),
\begin{equation}\label{5.37}
\cL H\le C_{56}+C_{57}H
+\frac{\beta}{k}\frac{\hat GJ}{\sigma_{k}}\sigma_{k}^{ij}(h_{ij})^2H.
\end{equation}
The next step is to get the evolution equation of $P=\hat GJ$.
\begin{equation}\label{5.38}
\begin{split}
\cL P=&\frac{P}{\hat G}\(d_X\hat G(X)(1-\frac{P}{u})
+d_\nu\hat G(u\nn(\frac{P}{u}))\)
+\frac{\beta}{k}\frac{P}{\sigma_{k}}\sigma_{k}^{ij}\(u\nn^2_{ij}(\frac{P}{u})-(1-\frac{P}{u})h^i_j\\
&+2h_{il}\nn^l(\frac{P}{u})\nn_j\Phi\)-\frac{\beta}{k}\frac{\hat GJ}{\sigma_{k}}\sigma_{k}^{ij}\nn_{ij}^2P-J(d_\nu \hat G(e_i)-\frac{\hat G}{u}\nn^i\Phi)\nn_iP\\
=&P(1-\frac{P}{u})\(\frac{d_X\hat G(X)}{\hat G}-\beta\)-\frac{P^2d_\nu\hat G(\frac{\nn u}{u})}{\hat G}-\frac{\beta}{k}\frac{P^2}{\sigma_{k}}\sigma_{k}^{ij}\frac{\nn^2_{ij}u}{u}+\frac{P}{u}\nn^i\Phi\nn_iP\\
=&P(1-\frac{P}{u})\(\frac{d_X\hat G(X)}{\hat G}-\beta\)-\beta\frac{P^2}{u}+\frac{\beta}{k}\frac{P^2}{\sigma_{k}}\sigma_{k}^{ij}h^2_{ij}+\frac{P^2}{u}\nn^i\Phi\frac{d_X\hat G(e_i)}{\hat G}\\
=&C_{58}+\frac{\beta}{k}\frac{P^2}{\sigma_{k}}\sigma_{k}^{ij}h^2_{ij}.
\end{split}
\end{equation}
Combining (\ref{5.37}) and (\ref{5.38}), we derive
\begin{equation}\label{5.39}
\begin{split}
\cL\theta=&\frac{\cL H}{P-a}-\frac{H\cL P}{(P-a)^2}\\
\le&C_{59}+C_{60}H
-\frac{\beta}{k}\frac{P}{P-a}\frac{a}{P-a}\frac{\sigma_{k}^{ij}}{\sigma_{k}}(h_{ij})^2H.
\end{split}
\end{equation}
Without loss of generality, we assume that $\k_1\ge\cdots\ge\k_n$ as before. By Proposition \ref{p2.6} and $H\le n\k_1$, we have
$$\sigma_{k}^{ij}(h_{ij})^2\ge \sigma_{k-1,1}\k_1^2\ge\frac{k}{n^2}\sigma_{k}H.$$
Thus
$$\cL\theta\le C_{59}+C_{60}H-C_{61}H^2,$$
which implies an upper bound of $\theta$. Hence, similar to Lemma \ref{l5.1}, $H$ is bounded along the flow (\ref{1.10}) and we complete this proof.
\end{proof}

\begin{proof of theorem 3.1 and 3.2}

The estimates obtained in Lemma \ref{l4.2}, \ref{l5.1}, \ref{l5.3}, Corollary \ref{c4.10} and \ref{c4.11} depend on $\beta$, $\Vert G\Vert_{C^0}$, $\Vert G\Vert_{C^1}$, $\Vert G\Vert_{C^2}$ and the geometry of the initial data $M_0$. They are independent of $T$. By Lemma \ref{l4.2}, \ref{l5.1}, \ref{l5.3}, Corollary \ref{c4.10} and \ref{c4.11}, we conclude that the equation (\ref{2.17}) is uniformly parabolic. By the $C^0$-estimate (Lemma \ref{l4.2}), the gradient estimate (Corollary \ref{c4.10} and \ref{c4.11}), the $C^2$-estimate (Lemma \ref{l5.1} and \ref{l5.3}) and the Krylov's and Nirenberg's theory \cite{KNV,LN}, we get the H$\ddot{o}$lder continuity of $D^2\rho$ and $\rho_t$. Then we can get higher order derivation estimates by the regularity theory of the uniformly parabolic equations. Hence we obtain the long time existence and $C^\infty$-smoothness of solutions for the flows (\ref{1.9}) and (\ref{1.10}). The uniqueness of smooth solutions also follows from the parabolic theory. 

By Lemma \ref{l4.2}, we can find the sign of $Q-1$ is preserved along flow (\ref{1.8}) for all $t\ge0$. Moreover, due to the $C^0$-estimate,
$$\vert \rho(\cdot,t)-\rho(\cdot,0)\vert=\vert\int_{0}^t\(\Psi(GF^{-\beta})-\Psi(1)\)\rho(\cdot,s)ds\vert<\infty.$$
Therefore, in view of the monotonicity of $\rho$, the limit $\rho(\cdot,\infty):=\lim_{t_i\rightarrow\infty}\rho(\cdot,t)$ exists and is positive, smooth and star-shaped. Thus the hypersurface with radial function $\rho(\cdot,\infty)$ is our desired solution to $$GF^{-\beta}=1.$$
This completes the proof of Theorems \ref{t3.1} and \ref{t3.2}.

\end{proof of theorem 3.1 and 3.2}


\section{The uniqueness results}
In the last section, we shall derive three uniqueness results.
\begin{theorem}\label{t6.1}
Suppose $F$ is a positive, homogeneous of degree $1$ and monotonically increasing function in $\Gamma$, $G(X,\nu)\in C^\infty(B_{r_2}\backslash B_{r_1}\times\mS^n)$ is a positive function if $\la X,\nu\ra>0$, and for any fixed unit vector $\nu$, $\frac{\p}{\p\rho}(\rho G)<0$. Then the star-shaped solution of (\ref{1.1}) is unique.
\end{theorem}
\begin{proof}
Let $X_1$ and $X_2$ are two star-shaped solutions of (\ref{1.1}). We have
\begin{equation}\label{6.1}
F(\k(X_l))=G(X_l,\nu(X_l)),\quad l=1,2.
\end{equation}
Let $\rho_l$ be the radial function of $X_l$, then $\g_l=\log \rho_l$. We shall prove that for any $\xi\in\mS^n$, we have
\begin{equation}\label{6.2}
\rho_1(\xi)\ge\rho_2(\xi).
\end{equation}
It is easy to see that (\ref{6.2}) would imply that $\g_1=\g_2$, and then $\rho_1=\rho_2$. To prove (\ref{6.2}), we define a function $\theta(\xi)=\g_1(\xi)-\g_2(\xi)$, and let $\xi_0\in\mS^n$ is the point where $\theta$ achieves its minimum. Then we have that $D\theta=0$ and $D^2\theta\ge0$ at $\xi_0$, that is $D\g_1=D\g_2$, $\om_1=\om_2$ and $D^2\g_1\ge D^2\g_2$.
By (\ref{6.1}), we have
\begin{align*}
F([\delta_{ij}-(e^{ik}-\frac{D_i\gamma_lD_k\gamma_l}{\omega^2})D^2_{kj}\gamma_{l}])=&\rho_l\om_lG(\rho_l\frac{\p}{\p\rho},\frac{1}{\om_l}(\frac{\p}{\p\rho}-\rho_l^{-2}D_i\rho_l\frac{\p}{\p x_i}))\\
=&\rho_l\om_lG(\rho_l\frac{\p}{\p\rho},\frac{1}{\om_l}(\frac{\p}{\p\rho}-D_i\g_le_i)),
\end{align*}
where $e_i=\rho^{-1}\frac{\p}{\p x_i}$ is an orthonormal basis on $\mS^n$.
 Note that $D\g_1=D\g_2$, $\om_1=\om_2$ and $D^2\g_1\ge D^2\g_2$. Thus we can let $\nu=\frac{1}{\om_l}(\frac{\p}{\p\rho}-D_i\g_le_i)$ be a unit vector. Then
\begin{align*}
\rho_1\om_1G(\rho_1\frac{\p}{\p\rho},\nu)=&
F([\delta_{ij}-(e^{ik}-\frac{D_i\gamma_1D_k\gamma_1}{\omega^2})D^2_{kj}\gamma_{1}])\\
\le& F([\delta_{ij}-(e^{ik}-\frac{D_i\gamma_2D_k\gamma_2}{\omega^2})D^2_{kj}\gamma_{2}])\\
=&\rho_2\om_2G(\rho_2\frac{\p}{\p\rho},\nu).
\end{align*}
In other words, $\rho_1G(X_1,\nu)\le \rho_2G(X_2,\nu)$, which implies that $\rho_1(\xi_0)\ge\rho_2(\xi_0)$ by $\frac{\p}{\p\rho}(\rho G)<0$. Thus we complete this proof.
\end{proof}
Finally we state two uniqueness results which are proved by flows.
	
	
	
	

	
	
	
	
	
\begin{corollary}\label{c6.2}
	Let $G^\frac{1}{\beta}=G^\frac{1}{\beta}(u,\rho)$ satisfy (\ref{1.2}), $F\in C^2(\Gamma_+)\cap C^0(\p\Gamma_+)$ satisfy Assumption \ref{a1.3}, and let $M_0$ be a closed, smooth, star-shaped hypersurface in $\mathbb{R}^{n+1}$, $n\ge2$, enclosing the origin. Suppose
	
	$(i)$ for any fixed $\rho>0$, $uG$ is log-convex with respect to $u$;
	
	$(ii)$ $\frac{uG_u}{G}\le\beta$;
	
	$(iii)$ $\frac{uG_u}{G}+\frac{\rho G_\rho}{G}+\beta\le0.$
	
	\noindent Then flow (\ref{1.9}) has a unique smooth, admissible  solution $M_t$ for all time $t>0$. And
	$M_t$ converges exponentially to a round sphere centered at the origin in the $C^\infty$-topology.
	
	This means that the solutions of $G(u,\rho)F^{-\beta}=1$ must be a sphere whose radius $R$ satisfies $G(R,R)R^\beta=1$.
\end{corollary}
\textbf{Remark}: If $G=u^{\a-1}\rho^\delta$, we could prove the long time existence and  convergence of this flow with $\a+\delta+\beta\le1$, $\a\le0$.
\begin{corollary}\label{c6.3}
	Let $\wg^{-\frac{1}{\beta}}=\wg^{-\frac{1}{\beta}}(u,\rho)$ satisfy (\ref{1.2}), $F=\sigma_{k}^\frac{1}{k}$, $2\le k\le n-1$, and let $M_0$ be a closed, smooth, star-shaped, strictly $k$-convex hypersurface in $\mathbb{R}^{n+1}$, $n\ge2$, enclosing the origin. Suppose
	
	$(i)$ either (\rmnum{1}a) for any fixed $\rho>0$, $u\wg$ is log-concave with respect to $u$ for $\beta\ge k$,
	
	or  (\rmnum{1}b) for any fixed $\rho>0$, $u\wg$ is $\frac{1}{1-\beta}$-concave with respect to $u$ for $0<\beta< 1$;
	
	$(ii)$ $\frac{u\wg_u}{\wg}+\beta\ge0$;
	
	$(iii)$ $\frac{u\wg_u}{\wg}+\frac{\rho \wg_\rho}{\wg}-\beta\ge0.$
	
	\noindent Then flow (\ref{1.10}) has a unique smooth, strictly $k$-convex solution $M_t$ for all time $t>0$. And
	$M_t$ converges exponentially to a round sphere centered at the origin in the $C^\infty$-topology.
	
	This means that the solutions of $\wg(u,\rho)\sigma_{k}^\frac{\beta}{k}=1$ must be a sphere whose radius $R$ satisfies $\wg(R,R)R^{-\beta}=1$.
\end{corollary}
\textbf{Remark}: If $\wg=u^{\a-1}\rho^\delta$, we could prove the long time existence and  convergence of this flow with $\a+\delta-\beta\ge1$, either $\a\ge0$, $\beta\ge k$ or $\a=1-\beta$, $0<\beta\le1$.

\begin{proof of corollaries 6.2 and 6.3}
It suffices to improve the $C^1$-estimate and convergence by the proof of Theorems \ref{t3.1} and \ref{t3.2}.  Recall the proof of Lemma \ref{l4.5} and note that $\wg=\frac{1}{G}$, we get according to (\ref{4.6}) and (\ref{4.7}),
\begin{align}
	\p_tO_{\max}\le&Q\Psi'\((\frac{uG_u}{G}+\frac{\rho G_\rho}{G}+\beta)|D\g|^2+\frac{\beta F^{ij}}{F}(-\g_{li}\g_{lj}+\g_i\g_j-\vert D\g\vert^2\delta_{ij})\)\notag\\
	\le&Q\Psi'\frac{\beta F^{ij}}{F}(\g_i\g_j-\vert D\g\vert^2\delta_{ij}).
\end{align}
Similar to Lemma \ref{l4.5}, since $F^{ij}\g_i\g_j\leq\max_iF^{ii}\vert D\g\vert^2$, we can derive the $C^1$-estimate by the maximum principle. Then we can get the a priori estimates by Section 5. Meanwhile we have the bound of $Q$, $F^{ii}$ and $F$, i.e. $\max_iF^{ii}-\sum_iF^{ii}\le -C_0$. This proves
	\begin{equation*}
		\max_{\mS^n}\frac{\vert D\rho(\cdot,t)\vert}{\rho(\cdot,t)}\leq Ce^{-C_0t},\forall t>0,
	\end{equation*}
	for both $C$ and $C_0$ are positive constants. In other words, we have that $\parallel D\rho\parallel\to0$ exponentially as $t\to\infty$. Hence by the interpolation and the a priori estimates, we can get that $\rho$ converges exponentially to a constant in the $C^\infty$-topology as $t\to\infty$. 
	By Theorems \ref{t3.1} and \ref{t3.2}, there exists a solution of $G(u,\rho)F^\beta=1$. Since this solution is invariant along flows (\ref{1.9}) or (\ref{1.10}), the solution must be a sphere whose radius $R$ satisfies $G(R,R)R^\beta=1$.
\end{proof of corollaries 6.2 and 6.3}

\section{Reference}



\begin{biblist}
\bib{AA}{article}{
	author={Aleksandrov A. D.},
	title={Uniqueness theorems for surfaces in the large. I},
	language={Russian},
	journal={Vestnik Leningrad. Univ.},
	volume={11},
	date={1956},
	number={19},
	pages={5--17},
	issn={0146-924x},
	review={\MR{0086338}},
}


\bib{B3}{article}{
   author={Andrews B.},
   title={Pinching estimates and motion of hypersurfaces by curvature functions},
   journal={J. Reine Angew. Math.},
   volume={608}
   date={2007},
   pages={17-33},
}




\bib{B4}{article}{
   author={Andrews B.},
   author={ McCoy J.},
   author={ Zheng Y.},
   title={Contracting convex hypersurfaces by curvature},
   journal={Calc. Var. PDEs },
   volume={47}
   date={2013},
   pages={611-665},
}







\bib{BIS}{article}{
	author={Bryan P.},
	author={Ivaki M. N.},
	author={Scheuer J.},
	title={A unified flow approach to smooth, even $L_p$-Minkowski problems},
	journal={Anal. PDE},
	volume={12},
	date={2019},
	number={2},
	pages={259--280},
	issn={2157-5045},
	review={\MR{3861892}},
	doi={10.2140/apde.2019.12.259},
}



\bib{BIS3}{article}{
	author={Bryan P.},
	author={Ivaki M. N.},
	author={Scheuer J.},
	title={Orlicz-Minkowski flows},
	journal={Calc. Var. Partial Differential Equations},
	volume={60},
	date={2021},
	number={1},
	pages={Paper No. 41, 25},
	issn={0944-2669},
	review={\MR{4204567}},
	doi={10.1007/s00526-020-01886-3},
}



\bib{CNS}{article}{
	author={Caffarelli L.},
	author={Nirenberg L.},
	author={Spruck J.},
	title={The Dirichlet problem for nonlinear second-order elliptic
		equations. I. Monge-Amp\`ere equation},
	journal={Comm. Pure Appl. Math.},
	volume={37},
	date={1984},
	number={3},
	pages={369--402},
	issn={0010-3640},
	review={\MR{739925}},
	doi={10.1002/cpa.3160370306},
}

\bib{CNS2}{article}{
	author={Caffarelli L.},
author={Nirenberg L.},
author={Spruck J.},
	title={Nonlinear second order elliptic equations. IV. Starshaped compact
		Weingarten hypersurfaces},
	conference={
		title={Current topics in partial differential equations},
	},
	book={
		publisher={Kinokuniya, Tokyo},
	},
	date={1986},
	pages={1--26},
	review={\MR{1112140}},
}




\bib{CHZ}{article}{
	author={Chen C.},
	author={Huang Y.},
	author={Zhao Y.},
	title={Smooth solutions to the $L_p$ dual Minkowski problem},
	journal={Math. Ann.},
	volume={373},
	date={2019},
	number={3-4},
	pages={953--976},
	issn={0025-5831},
	review={\MR{3953117}},
	doi={10.1007/s00208-018-1727-3},
}



\bib{CCL}{article}{
	author={Chen H.},
	author={Chen S.},
	author={Li Q.},
	title={Variations of a class of Monge-Amp\`ere-type functionals and their
		applications},
	journal={Anal. PDE},
	volume={14},
	date={2021},
	number={3},
	pages={689--716},
	issn={2157-5045},
	review={\MR{4259871}},
	doi={10.2140/apde.2021.14.689},
}



\bib{CL}{article}{
	author={Chen H.},
	author={Li Q.},
	title={The $L_ p$ dual Minkowski problem and related parabolic flows},
	journal={J. Funct. Anal.},
	volume={281},
	date={2021},
	number={8},
	pages={Paper No. 109139, 65},
	issn={0022-1236},
	review={\MR{4271790}},
	doi={10.1016/j.jfa.2021.109139},
}






\bib{CLW}{article}{
	author={Chen D.},
	author={Li H.},
	author={Wang Z.},
	title={Starshaped compact hypersurfaces with prescribed Weingarten
		curvature in warped product manifolds},
	journal={Calc. Var. Partial Differential Equations},
	volume={57},
	date={2018},
	number={2},
	pages={Paper No. 42, 26},
	issn={0944-2669},
	review={\MR{3764587}},
	doi={10.1007/s00526-018-1314-1},
}



\bib{DL2}{article}{
	author={Ding S.},
	author={Li G.},
	title={A class of curvature flows expanded by support function and
		curvature function in the Euclidean space and hyperbolic space},
	journal={J. Funct. Anal.},
	volume={282},
	date={2022},
	number={3},
	pages={Paper No. 109305},
	issn={0022-1236},
	review={\MR{4339010}},
	doi={10.1016/j.jfa.2021.109305},
}

\bib{DL3}{article}{
	author={Ding S.},
author={Li G.},
	title={A class of inverse curvature flows and $L^p$ dual
		Christoffel-Minkowski problem},
	journal={Trans. Amer. Math. Soc.},
	volume={376},
	date={2023},
	number={1},
	pages={697--752},
	issn={0002-9947},
	review={\MR{4510121}},
	doi={10.1090/tran/8793},
}

\bib{DL4}{article}{
	author={Ding S.},
	author={Li G.},
	title={Anisotropic flows without global terms and dual Orlicz Christoffel-Minkowski type problem},
	pages={arXiv:2207.03114},
}



\bib{EH}{article}{
	author={Ecker K.},
	author={Huisken G.},
	title={Immersed hypersurfaces with constant Weingarten curvature},
	journal={Math. Ann.},
	volume={283},
	date={1989},
	number={2},
	pages={329--332},
	issn={0025-5831},
	review={\MR{980601}},
	doi={10.1007/BF01446438},
}



\bib{GHW1}{article}{
	author={Gardner J.},
	author={Hug D.},
	author={Weil W.},
	author={Xing S.},
	author={Ye D.},
	title={General volumes in the Orlicz-Brunn-Minkowski theory and a related
		Minkowski problem I},
	journal={Calc. Var. Partial Differential Equations},
	volume={58},
	date={2019},
	number={1},
	pages={Paper No. 12, 35},
	issn={0944-2669},
	review={\MR{3882970}},
	doi={10.1007/s00526-018-1449-0},
}

\bib{GHW2}{article}{
	author={Gardner J.},
	author={Hug D.},
	author={Xing S.},
	author={Ye D.},
	title={General volumes in the Orlicz-Brunn-Minkowski theory and a related
		Minkowski problem II},
	journal={Calc. Var. Partial Differential Equations},
	volume={59},
	date={2020},
	number={1},
	pages={Paper No. 15, 33},
	issn={0944-2669},
	review={\MR{4040624}},
	doi={10.1007/s00526-019-1657-2},
}




\bib{GC2}{book}{
	author={Gerhardt C.},
	title={Curvature problems},
	series={Series in Geometry and Topology},
	volume={39},
	publisher={International Press, Somerville, MA},
	date={2006},
	pages={x+323},
	isbn={978-1-57146-162-9},
	isbn={1-57146-162-0},
	review={\MR{2284727}},
}




\bib{GG}{article}{
	author={Guan B.},
	author={Guan P.},
	title={Convex hypersurfaces of prescribed curvatures},
	journal={Ann. of Math. (2)},
	volume={156},
	date={2002},
	number={2},
	pages={655--673},
	issn={0003-486X},
	review={\MR{1933079}},
	doi={10.2307/3597202},
}





\bib{GL}{article}{
	author={Guan P.},
	author={Li J.},
	title={A mean curvature type flow in space forms},
	journal={Int. Math. Res. Not. IMRN},
	date={2015},
	number={13},
	pages={4716--4740},
	issn={1073-7928},
	review={\MR{3439091}},
	doi={10.1093/imrn/rnu081},
}



\bib{GLL}{article}{
	author={Guan P.},
author={Li J.},
	author={Li Y.},
	title={Hypersurfaces of prescribed curvature measure},
	journal={Duke Math. J.},
	volume={161},
	date={2012},
	number={10},
	pages={1927--1942},
	issn={0012-7094},
	review={\MR{2954620}},
	doi={10.1215/00127094-1645550},
}

\bib{GLM}{article}{
	author={Guan P.},
	author={Lin C.},
	author={Ma X.-N.},
	title={The existence of convex body with prescribed curvature measures},
	journal={Int. Math. Res. Not. IMRN},
	date={2009},
	number={11},
	pages={1947--1975},
	issn={1073-7928},
	review={\MR{2507106}},
	doi={10.1093/imrn/rnp007},
}

\bib{GM}{article}{
	author={Guan P.},
	author={Ma X.},
	title={The Christoffel-Minkowski problem. I. Convexity of solutions of a
		Hessian equation},
	journal={Invent. Math.},
	volume={151},
	date={2003},
	number={3},
	pages={553--577},
	issn={0020-9910},
	review={\MR{1961338}},
	doi={10.1007/s00222-002-0259-2},
}



\bib{GRW}{article}{
	author={Guan P.},
	author={Ren C.},
	author={Wang Z.},
	title={Global $C^2$-estimates for convex solutions of curvature
		equations},
	journal={Comm. Pure Appl. Math.},
	volume={68},
	date={2015},
	number={8},
	pages={1287--1325},
	issn={0010-3640},
	review={\MR{3366747}},
	doi={10.1002/cpa.21528},
}

\bib{GX}{article}{
	author={Guan P.},
	author={Xia C.},
	title={$L^p$ Christoffel-Minkowski problem: the case $1<p<k+1$},
	journal={Calc. Var. Partial Differential Equations},
	volume={57},
	date={2018},
	number={2},
	pages={Paper No. 69, 23},
	issn={0944-2669},
	review={\MR{3776359}},
	doi={10.1007/s00526-018-1341-y},
}

\bib{HMS}{article}{
	author={Hu C.},
	author={Ma X.},
	author={Shen C.},
	title={On the Christoffel-Minkowski problem of Firey's $p$-sum},
	journal={Calc. Var. Partial Differential Equations},
	volume={21},
	date={2004},
	number={2},
	pages={137--155},
	issn={0944-2669},
	review={\MR{2085300}},
	doi={10.1007/s00526-003-0250-9},
}




\bib{HLY2}{article}{
	author={Huang Y.},
	author={Lutwak E.},
	author={Yang D.},
	author={Zhang G.},
	title={Geometric measures in the dual Brunn-Minkowski theory and their
		associated Minkowski problems},
	journal={Acta Math.},
	volume={216},
	date={2016},
	number={2},
	pages={325--388},
	issn={0001-5962},
	review={\MR{3573332}},
	doi={10.1007/s11511-016-0140-6},
}




\bib{HZ}{article}{
	author={Huang Y.},
	author={Zhao Y.},
	title={On the $L_p$ dual Minkowski problem},
	journal={Adv. Math.},
	volume={332},
	date={2018},
	pages={57--84},
	issn={0001-8708},
	review={\MR{3810248}},
	doi={10.1016/j.aim.2018.05.002},
}






\bib{HGC}{article}{
   author={Huisken G.},
   author={Sinestrari C.},
   title={Convexity estimates for mean curvature flow and singularities of mean convex surfaces},
   journal={Acta Math.},
   volume={183}
   date={1999},
   pages={45-70},
}





\bib{IM}{article}{
   author={Ivaki M.},
   title={Deforming a hypersurface by principal radii of curvature and support function},
   journal={Calc. Var. PDEs},
   volume={58(1)}
   date={2019},
}

\bib{JL}{article}{
	author={Jin Q.},
	author={Li Y.},
	title={Starshaped compact hypersurfaces with prescribed $k$-th mean
		curvature in hyperbolic space},
	journal={Discrete Contin. Dyn. Syst.},
	volume={15},
	date={2006},
	number={2},
	pages={367--377},
	issn={1078-0947},
	review={\MR{2199434}},
	doi={10.3934/dcds.2006.15.367},
}



\bib{KA}{article}{
	author={Kennington A. U.},
	title={Power concavity and boundary value problems},
	journal={Indiana Univ. Math. J.},
	volume={34},
	date={1985},
	number={3},
	pages={687--704},
	issn={0022-2518},
	review={\MR{794582}},
	doi={10.1512/iumj.1985.34.34036},
}



\bib{KNV}{book}{
  author={Krylov N. V.},
     title= {Nonlinear elliptic and parabolic quations of the second order},
 publisher={D. Reidel Publishing Co., Dordrecht},
     date={1987. xiv+462pp},

}

\bib{LRW}{article}{
	author={Li C.},
	author={Ren C.},
	author={Wang Z.},
	title={Curvature estimates for convex solutions of some fully nonlinear
		Hessian-type equations},
	journal={Calc. Var. Partial Differential Equations},
	volume={58},
	date={2019},
	number={6},
	pages={Paper No. 188, 32},
	issn={0944-2669},
	review={\MR{4024603}},
	doi={10.1007/s00526-019-1623-z},
}


\bib{LN}{book}{
  author={Nirenberg L.},
     title= {On a generalization of quasi-conformal mappings and its application to elliptic partial differential equations},
 publisher={Contributions to the theory of partial differential equations, Annals of Mathematics Studies},
     date={ Princeton University Press, Princeton, N. J.,1954, pp. 95C100.}
  }



\bib{LXZ}{article}{
	author={Li H.},
	author={Xu B.},
	author={Zhang R.},
	title={Asymptotic convergence for a class of anisotropic curvature flows},
	journal={J. Funct. Anal.},
	volume={282},
	date={2022},
	number={12},
	pages={Paper No. 109460, 34},
	issn={0022-1236},
	review={\MR{4395335}},
	doi={10.1016/j.jfa.2022.109460},
}

\bib{LSW}{article}{
   author={Li Q.},
   author={Sheng W.},
   author={Wang X-J},
   title={Flow by Gauss curvature to the Aleksandrov and dual Minkowski problems},
   journal={Journal of the European Mathematical Society},
   volume={22}
   date={2019},
   pages={893-923},
}

\bib{LSW2}{article}{
   author={Li Q.},
author={Sheng W.},
author={Wang X-J},
	title={Asymptotic convergence for a class of fully nonlinear curvature
		flows},
	journal={J. Geom. Anal.},
	volume={30},
	date={2020},
	number={1},
	pages={834--860},
	issn={1050-6926},
	review={\MR{4058539}},
	doi={10.1007/s12220-019-00169-4},
}




\bib{LL}{article}{
	author={Liu Y.},
	author={Lu J.},
	title={A flow method for the dual Orlicz-Minkowski problem},
	journal={Trans. Amer. Math. Soc.},
	volume={373},
	date={2020},
	number={8},
	pages={5833--5853},
	issn={0002-9947},
	review={\MR{4127893}},
	doi={10.1090/tran/8130},
}

\bib{LE}{article}{
	author={Lutwak E.},
	title={The Brunn-Minkowski-Firey theory. I. Mixed volumes and the
		Minkowski problem},
	journal={J. Differential Geom.},
	volume={38},
	date={1993},
	number={1},
	pages={131--150},
	issn={0022-040X},
	review={\MR{1231704}},
}



\bib{LO}{article}{
	author={Lutwak E.},
	author={Oliker V.},
	title={On the regularity of solutions to a generalization of the
		Minkowski problem},
	journal={J. Differential Geom.},
	volume={41},
	date={1995},
	number={1},
	pages={227--246},
	issn={0022-040X},
	review={\MR{1316557}},
}


\bib{LYZ2}{article}{
	author={Lutwak E.},
	author={Yang D.},
	author={Zhang G.},
	title={$L_p$ dual curvature measures},
	journal={Adv. Math.},
	volume={329},
	date={2018},
	pages={85--132},
	issn={0001-8708},
	review={\MR{3783409}},
	doi={10.1016/j.aim.2018.02.011},
}

\bib{RW1}{article}{
	author={Ren C.},
	author={Wang Z.},
	title={On the curvature estimates for Hessian equations},
	journal={Amer. J. Math.},
	volume={141},
	date={2019},
	number={5},
	pages={1281--1315},
	issn={0002-9327},
	review={\MR{4011801}},
	doi={10.1353/ajm.2019.0033},
}

\bib{RW2}{article}{
	author={Ren C.},
	author={Wang Z.},
	title={The global curvature estimate for the $n-2$ Hessian equation},
	pages={arXiv:2002.08702},
}




\bib{SR}{book}{
	author={Schneider R.},
	title={Convex bodies: the Brunn-Minkowski theory},
	series={Encyclopedia of Mathematics and its Applications},
	volume={151},
	edition={Second expanded edition},
	publisher={Cambridge University Press, Cambridge},
	date={2014},
	pages={xxii+736},
	isbn={978-1-107-60101-7},
	review={\MR{3155183}},
}






\bib{SX}{article}{
	author={Spruck J.},
	author={Xiao L.},
	title={A note on star-shaped compact hypersurfaces with prescribed scalar
		curvature in space forms},
	journal={Rev. Mat. Iberoam.},
	volume={33},
	date={2017},
	number={2},
	pages={547--554},
	issn={0213-2230},
	review={\MR{3651014}},
	doi={10.4171/RMI/948},
}


\bib{TW}{article}{
	author={Treibergs A. E.},
	author={Wei S. W.},
	title={Embedded hyperspheres with prescribed mean curvature},
	journal={J. Differential Geom.},
	volume={18},
	date={1983},
	number={3},
	pages={513--521},
	issn={0022-040X},
	review={\MR{723815}},
}





\bib{UJ}{article}{
   author={Urbas J.},
   title={An expansion of convex hypersurfaces},
   journal={J. Diff. Geom.},
   volume={33(1)}
   date={1991},
   pages={91-125},
}





\end{biblist}
\end{document}